\documentclass[12pt,reqno]{amsart}
\usepackage{amsmath,amsfonts,amsthm,amsopn,amssymb,mathrsfs,enumerate,color}
\usepackage{amsmath}
\usepackage{comment}

\usepackage{cite,marginnote}
\usepackage{mathtools} % ��Ҫ���� mathtools ���

\usepackage{color,graphicx,enumerate}
\usepackage[colorlinks=true,urlcolor=blue,
citecolor=red,linkcolor=blue,linktocpage,pdfpagelabels,
bookmarksnumbered,bookmarksopen]{hyperref}
\usepackage[english]{babel}

\usepackage[left=2.9cm,right=2.9cm,top=2.8cm,bottom=2.8cm]{geometry}
\numberwithin{equation}{section}
\makeindex
\newtheorem{theorem}{Theorem}[section]

\newtheorem{corollary}{Corollary}[section]

\newtheorem{proposition}{Proposition}[section]
\newtheorem{lemma}{Lemma}[section]

\newtheorem{iteration lemma}{iteration Lemma}[section]

\usepackage{amsmath,amssymb,amsthm,mathtools}
\usepackage{hyperref}

\usepackage{tikz}
\usetikzlibrary{arrows.meta,decorations.pathreplacing,positioning}

\theoremstyle{remark}

\title[Stability for Affine Sobolev Inequality]{Sharp Quantitative Stability for the Affine \(p\)-Sobolev Inequality, Part II: The Case \(1<p<2\)}

\author[S.~Fan]{Song Fan}
\author[G-D.~Li]{Gui-Dong Li}
\author[J.~J.~Zhang]{Jianjun Zhang}
	
\address[S.~Fan]{\newline\indent School  of  Mathematics  and  Statistics
\newline\indent
Guizhou University
\newline\indent
Guiyang, 550025, Guizhou, PR China}
\email{\href{mailto:doraemonsong77@gmail.com}{doraemonsong77@gmail.com}}

\address[G-D.~Li]{\newline\indent School  of  Mathematics  and  Statistics
\newline\indent
Guizhou University
\newline\indent
Guiyang, 550025, Guizhou, PR China}
\email{\href{mailto:bestdong123@163.com}{bestdong123@163.com}}

\address[J.~J.~Zhang]{\newline\indent College of Mathematics and Statistics
\newline\indent
Chongqing Jiaotong University
\newline\indent
Xuefu, Nan'an, 400074, Chongqing, PR China}
\email{\href{mailto:zhangjianjun09@tsinghua.org.cn}{zhangjianjun09@tsinghua.org.cn}}

\begin{document}

\begin{abstract}

  We prove a sharp quantitative stability  result for the affine \(L^p\)-Sobolev inequality, for
  \(1<p<2\), introduced by Lutwak--Yang--Zhang (\emph{J. Differential Geom.}, \textbf{62} (2002),
  17--38). Moreover, the stability exponent is shown to be optimal, and equal to \(2\).

  \vskip0.23in

  \noindent{\it   {\bf Key  words:} Sharp quantitative stability, Affine \(L^p\) Sobolev inequality, Optimal stability exponent.}


%\vskip0.1in
%\noindent{\it  {\bf 2010 Mathematics Subject Classification:}} 35A16, 35B40, 35A02, 35J60, 46N50.
\end{abstract}

\maketitle

\section{Introduction}

Let \(1<p<n\), \(p^*=np/(n-p)\), and \(p'=p/(p-1)\). We work with real-valued functions in the
homogeneous Sobolev space \(\dot W^{1,p}(\mathbb R^n)\), identified with the functions in
\(L^{p^*}(\mathbb R^n)\) whose distributional gradient belongs to \(L^p(\mathbb R^n)\). The sharp
Sobolev inequality \cite{Aubin76,Talenti76} is
\begin{equation}
  \label{eq:classical-Sobolev}
  S_{n,p}^p\|u\|_{L^{p^*}(\mathbb R^n)}^p \le \|\nabla u\|_{L^p(\mathbb R^n)}^p, \qquad u\in\dot W^{1,p}(\mathbb R^n).
\end{equation}
The nonzero extremals are obtained from \(U(x)=(1+|x|^{p'})^{-(n-p)/p}\) by multiplication by a nonzero
constant, translation, and dilation. We fix this choice of \(U\) throughout.

Brezis and Lieb \cite{BL85} raised the question of remainder terms in the Sobolev inequality. For
\(p=2\) and \(n\ge3\), Bianchi and Egnell \cite{BE91} proved that the deficit controls the square of the
distance to the extremal manifold in the gradient norm. For general \(1<p<n\), Cianchi, Fusco, Maggi,
and Pratelli \cite{CFMP09} obtained a quantitative estimate in the \(L^{p^*}(\mathbb R^n)\)-distance to
the extremal manifold. Figalli and Neumayer \cite{FN19} established gradient stability for \(2\le p<n\),
and Neumayer \cite{Neumayer20} obtained strong-form stability throughout \(1<p<n\). Figalli and Zhang
\cite{FZ22} determined the optimal exponent \(\max\{2,p\}\) in the relative gradient distance.  In the
Hilbertian case \(p=2\), Dolbeault, Esteban, Figalli, Frank, and Loss \cite{DEFFL25} obtained explicit
stability constants with optimal dimensional dependence.

At the endpoint \(p=1\), Fusco, Maggi, and Pratelli \cite{FMP07} proved quantitative stability for the
\(BV\)-Sobolev inequality, whose extremals are multiples of characteristic functions of balls. Figalli,
Maggi, and Pratelli \cite{FMP13} treated the anisotropic \(BV\)-Sobolev and logarithmic Sobolev
inequalities, using a distance that also contains a truncated variation term. These endpoint distances
differ from the gradient distance for \(1<p<n\).

For the Hilbertian fractional Sobolev inequality of order \(0<s<n/2\), Chen, Frank, and Weth
\cite{CFW13} proved a remainder estimate in the square of the homogeneous fractional Sobolev distance.
K\"onig \cite{Konig23} showed that the best global stability constant is strictly smaller than the
constant obtained from the local spectral gap. For \(n\ge3\) and \(0<s<1\), Chen, Lu, and Tang
\cite{CLT25} obtained explicit lower bounds with optimal asymptotic dependence on the dimension and the
fractional order. They also proved a global logarithmic Sobolev stability estimate on the sphere.

For related interpolation inequalities, Carlen and Figalli \cite{CF13} proved stability for the planar
Gagliardo--Nirenberg--Sobolev inequality involving \(\|u\|_{L^6(\mathbb R^2)}\), \(\|\nabla
u\|_{L^2(\mathbb R^2)}\), and \(\|u\|_{L^4(\mathbb R^2)}\), and for the logarithmic
Hardy--Littlewood--Sobolev inequality, with an application to the critical-mass Keller--Segel equation.
Nguyen \cite{Nguyen19} treated the Del~Pino--Dolbeault inequalities interpolating \(L^{2t}\) between
\(\dot H^1\) and \(L^{t+1}\), for \(n\ge2\) and \(1<t<(2n+1)/(2n-3)\). Zhang and Zhang \cite{ZZ25}
extended the superlinear range to \(1<t<n/(n-2)\) when \(n\ge3\), and to all \(t>1\) when \(n=2\), and
also treated the corresponding sublinear family \(1/2<t<1\). Here the gradient remains in \(L^2\); the
parameter \(t\) specifies the Lebesgue exponents. For the ordinary Hardy--Littlewood--Sobolev
inequality, Chen, Lu, and Tang \cite{CLT24} proved stability with explicit positive lower bounds for the
constant in the diagonal case, measured by the squared \(L^{2n/(n+2s)}(\mathbb R^n)\)-distance,
\(0<s<n/2\).

For the classical half-space trace inequality with \(p=2\) and \(n\ge3\), Ho \cite{Ho22} proved a
Bianchi--Egnell estimate in the gradient norm. Zhang, Zhou, and Zou \cite{ZZZ25} established quadratic
stability for the Hilbertian fractional trace inequality: for a trace of codimension \(m\), their range
is \(m/2<\alpha<n/2\), and the distance is measured in \(\dot H^\alpha\). Ma, Zhang, and Zhou
\cite{MZZ26} obtained gradient stability for the classical half-space trace inequality when \(n\ge3\)
and \(1<p<n\), with optimal exponent \(\max\{2,p\}\).

The affine Sobolev inequality strengthens \eqref{eq:classical-Sobolev}. Zhang \cite{Zhang99} proved the
endpoint inequality for \(p=1\). For \(1<p<n\), Lutwak, Yang, and Zhang \cite{LYZ02} deduced the sharp
affine \(L^p\)-Sobolev inequality from the \(L^p\) affine isoperimetric inequalities in \cite{LYZ00}.
For \(u\in\dot W^{1,p}(\mathbb R^n)\) and \(\xi\in\mathbb S^{n-1}\), define
\[
  A_\xi(u)=\int_{\mathbb R^n}|\partial_\xi u|^p\,dx, \qquad \partial_\xi u=\xi\cdot\nabla u,
\]
and write
\[
  \langle g\rangle_{\mathbb S^{n-1}} = \frac1{|\mathbb S^{n-1}|}\int_{\mathbb S^{n-1}}g(\xi)\,d\sigma(\xi), \qquad m_{n,p}=\bigl\langle|\omega_1|^p\bigr\rangle_{\omega\in\mathbb S^{n-1}}.
\]
The affine energy is
\begin{equation}
  \label{eq:def-E}
  \mathcal E_{\rm aff,p}(u) = m_{n,p}^{-1/p} \left\langle A_\xi(u)^{-n/p}\right\rangle_{\mathbb S^{n-1}}^{-1/n}.
\end{equation}
With the normalization in \eqref{eq:def-E}, one has
\begin{equation}
  \label{eq:affine-Sobolev}
  S_{n,p}^p\|u\|_{L^{p^*}(\mathbb R^n)}^p \le \mathcal E_{\rm aff,p}(u)^p \le \|\nabla u\|_{L^p(\mathbb R^n)}^p.
\end{equation}
If \(u\) is radial, then \(\mathcal E_{\rm aff,p}(u)=\|\nabla u\|_{L^p(\mathbb R^n)}\). By
\eqref{eq:affine-Sobolev}, the affine deficit is bounded above by the classical Sobolev deficit.
Classical Sobolev stability therefore does not directly imply the estimate sought here. The extremal
family also contains nonorthogonal linear deformations of \(U\).

The equality cases in the first inequality of \eqref{eq:affine-Sobolev} are the affine images of the
Aubin--Talenti functions; see \cite{LYZ02} and \cite{HJM16}. Including zero, the set of extremals is
\[
  \mathcal M_{\rm aff} = \left\{ c\,\lambda^{\frac{n-p}{p}} U\bigl(\lambda A(\,\cdot-x_0)\bigr): c\in\mathbb R,\ \lambda>0,\ x_0\in\mathbb R^n,\ A\in SL(n) \right\}.
\]
For \(\lambda>0\), \(A\in SL(n)\), and \(x_0\in\mathbb R^n\), set
\[
  (T_{\lambda A,x_0}u)(x) = \lambda^{-\frac{n-p}{p}} u\bigl(\lambda^{-1}A^{-1}x+x_0\bigr).
\]
These transformations preserve the \(L^{p^*}(\mathbb R^n)\)-norm and the affine energy. The gradient
norm is invariant under translations and Sobolev dilations, but generally changes under \(A\in SL(n)\).
We therefore measure the gradient error after optimizing over these transformations.

Several extensions of the sharp affine inequality are known. Haberl and Schuster \cite{HS09} proved
asymmetric affine \(L^p\)-Sobolev inequalities using positive directional derivatives. Haberl, Schuster,
and Xiao \cite{HSX12} established the asymmetric affine P\'olya--Szeg\H{o} principle. Cianchi, Lutwak,
Yang, and Zhang \cite{CLYZ09} proved affine Moser--Trudinger inequalities at \(p=n\) and affine
Morrey--Sobolev inequalities for \(p>n\). Haddad, Jim\'enez, and Montenegro \cite{HJM16} gave a proof of
\eqref{eq:affine-Sobolev} based on the \(L^p\) Busemann--Petty centroid inequality, and later proved
weighted affine inequalities on a half-space with a monomial weight \cite{HJM19}. De~N\'apoli, Haddad,
Jim\'enez, and Montenegro \cite{DHJM18} treated the sharp affine \(L^2\) trace inequality and its
\(L^p\) variants. Haddad and Ludwig \cite{HL24} proved affine fractional \(L^p\)-Sobolev inequalities
and asymmetric versions. These results establish sharp inequalities and their affine extensions; the
quantitative stability question requires an additional estimate for the deficit.

\subsection*{Main result and comparison with earlier work}

The earlier quantitative affine results concern different deficits or distances. Wang \cite{Wang13}
proved a quantitative affine P\'olya--Szeg\H{o} principle in terms of \(L^1(\mathbb R^n)\)-distance to
ellipsoidal rearrangements. The estimate assumes support of finite measure and includes an additional
term measuring the region where the gradient is small. Nguyen \cite[Theorem~1.2]{Nguyen16} proved
stability for the endpoint affine Sobolev inequality on \(BV\), in terms of a power of the normalized
\(L^{n/(n-1)}(\mathbb R^n)\)-distance to multiples of ellipsoid indicators; the exponent in that
estimate is not optimal. Neither result gives the gradient-distance estimate considered here for
\(1<p<2\). In \cite{FLZ26A}, we proved sharp stability for the affine fractional \(L^2\)-Sobolev
inequality. For the local inequality \eqref{eq:affine-Sobolev}, we use the distance
\begin{equation}
  \label{eq:affine-distance}
  d_{\rm aff}(u,\mathcal M_{\rm aff}) = \inf_{\substack{a\in\mathbb R,\ A\in SL(n), \\
  \lambda>0,\ x_0\in\mathbb R^n}} \frac{ \left\| \nabla\left(T_{\lambda A,x_0}u-aU\right) \right\|_{L^p(\mathbb R^n)} }{ \left\| \nabla T_{\lambda A,x_0}u \right\|_{L^p(\mathbb R^n)} },
\end{equation}
for \(u\in\dot W^{1,p}(\mathbb R^n)\setminus\{0\}\). The distance is invariant under \(T_{\lambda
A,x_0}\) and multiplication by a nonzero constant. Taking \(a=0\) gives \(0\le d_{\rm aff}\le1\).

In Part~I \cite{FLZ26B}, we proved
\[
  \frac{\mathcal E_{\rm aff,p}(u)} {S_{n,p}\|u\|_{L^{p^*}(\mathbb R^n)}} -1 \ge c_{n,p}\, d_{\rm aff}(u,\mathcal M_{\rm aff})^p, \qquad 2\le p<n,
\]
with optimal exponent \(p\). In a concurrent and independent work, Frank, Li, and Yang
\cite[Theorem~1.1]{FLY26} proved, for \(2\le p<n\), the homogeneous estimate
\[
  \begin{aligned}
    &\mathcal E_{\rm aff,p}(u)^p-S_{n,p}^p\|u\|_{L^{p^*}(\mathbb R^n)}^p \\
    &\quad\ge c_{n,p} \inf_{\substack{v\in\mathcal M_{\rm aff}\setminus\{0\} \\
    B\in SL(n)}} \left\{ \|\nabla(u\circ B-v)\|_{L^p(\mathbb R^n)}^p +\int_{\mathbb R^n}|\nabla v|^{p-2} |\nabla(u\circ B-v)|^2\,dx \right\}.
  \end{aligned}
\]
They also proved quantitative stability estimates for critical points of the affine Sobolev functional.

In the present paper, we prove the corresponding gradient-distance estimate for \(1<p<2\). In this
range, the optimal exponent is \(2\). Thus the optimal exponent agrees with that in the classical
Sobolev inequality proved by Figalli--Zhang \cite{FZ22}.

Our main result is the following.

\begin{theorem}
  \label{thm:intro-main-affine-p<2}
  Let \(n\ge 2\) and \(1<p<2\). There exists \(c_{n,p}>0\) such that every \(u\in\dot W^{1,p}(\mathbb
  R^n)\setminus\{0\}\) satisfies
  \begin{equation}
    \label{eq:intro-main-affine-p<2}
    \frac{\mathcal E_{\rm aff,p}(u)} {S_{n,p}\|u\|_{L^{p^*}(\mathbb R^n)}}-1 \ge c_{n,p}\,d_{\rm aff}(u,\mathcal M_{\rm aff})^2.
  \end{equation}
  The exponent \(2\) is optimal.
\end{theorem}

By Theorem~\ref{thm:intro-main-affine-p<2} and Part~I \cite{FLZ26B}, one has the following corollary.

\begin{corollary}
  Let \(1<p<n\). There exists \(c_{n,p}>0\) such that every \(u\in\dot W^{1,p}(\mathbb
  R^n)\setminus\{0\}\) satisfies
  \[
    \frac{\mathcal E_{\rm aff,p}(u)} {S_{n,p}\|u\|_{L^{p^*}(\mathbb R^n)}}-1 \ge c_{n,p}\, d_{\rm aff}(u,\mathcal M_{\rm aff})^{\max\{2,p\}}.
  \]
  The exponent \(\max\{2,p\}\) is optimal.
\end{corollary}

\subsection*{Strategy of the proof}

We first prove
\begin{equation}
  \label{eq:strategy-local-bound}
  \mathcal E_{\rm aff,p}(U+\varepsilon\phi)^p -S_{n,p}^p\|U+\varepsilon\phi\|_{L^{p^*}(\mathbb R^n)}^p \ge c\varepsilon^2
\end{equation}
for \(0<\varepsilon\le\varepsilon_0\), \(\|\nabla\phi\|_{L^p(\mathbb R^n)}=1\), and \(\phi\perp
T_U\mathcal M_{\rm aff}\). The constants \(c,\varepsilon_0>0\) depend only on \(n,p\). The notation
\(\phi\perp T_U\mathcal M_{\rm aff}\) means
\[
  \int_{\mathbb R^n}U^{p^*-2}Z\phi\,dx=0 \qquad \text{for every }Z\in T_U\mathcal M_{\rm aff}.
\]

For each fixed smooth, compactly supported \(\phi\) that is constant near the origin, the second
variation gives
\[
  \mathcal E_{\rm aff,p}(U+\varepsilon\phi)^p -S_{n,p}^p\|U+\varepsilon\phi\|_{L^{p^*}(\mathbb R^n)}^p =\frac p2\varepsilon^2Q_{\rm aff,p}(\phi) +o(\varepsilon^2).
\]
For \(n\ge3\), we verify that the proof of \cite[Section~5]{FLZ26B} applies to \(1<p<2\) and yields
\(\ker_{\mathcal H_U}Q_{\rm aff,p} =T_U\mathcal M_{\rm aff}\). For \(n=2\), we prove the nondegeneracy
using Fourier series in Appendix~\ref{app:kernel-affine-Hessian-n=2}.
Proposition~\ref{prop:affine-spectral-gap} then gives
\[
  Q_{\rm aff,p}(\phi) \ge c_{\rm sg}\int_{\mathbb R^n} |\nabla U|^{p-2}|\nabla\phi|^2\,dx
\]
for every \(\phi\in\mathcal H_U\) with \(\phi\perp T_U\mathcal M_{\rm aff}\). The space \(\mathcal H_U\)
is defined in \eqref{eq:def-HU-low}.

The assumption \(\phi\in\dot W^{1,p}\) does not ensure that \(\int_{\mathbb R^n}|\nabla
U|^{p-2}|\nabla\phi|^2\,dx\) is finite. Also, the expansion for a fixed \(\phi\) does not give constants
valid for all \(\phi\) with \(\|\nabla\phi\|_{L^p(\mathbb R^n)}=1\). We therefore estimate
\(\|\partial_\xi(U+\varepsilon\phi)\|_{L^p(\mathbb R^n)}^p\) using \cite[Lemma~2.1]{FZ22}. The
inequality includes the nonnegative term
\[
  \min\left\{ \varepsilon^p|\partial_\xi\phi|^p,\, \varepsilon^2|\partial_\xi U|^{p-2} |\partial_\xi\phi|^2 \right\}.
\]
Its integral is finite because the minimum is bounded by \(\varepsilon^p|\partial_\xi\phi|^p\). Applying
the directional estimates in the formula for \(\mathcal E_{\rm aff,p}\) gives
Lemma~\ref{lem:affine-nonlinear-lower-p<2}.

We next estimate \(\|U+\varepsilon\phi\|_{L^{p^*}(\mathbb R^n)}^p\). For \(p^*\le2\), we use
\cite[Lemma~2.4(i)]{FZ22} in terms of \(\mathcal P_\varepsilon(\phi)\), defined in
\eqref{sec2edit:def-Pepsilon}. This estimate remains applicable when \(\int_{\mathbb
R^n}U^{p^*-2}\phi^2\,dx\) is infinite. For \(p^*>2\), we use \cite[Lemma~2.4(ii)]{FZ22}. The estimate
involves \(\int_{\mathbb R^n}U^{p^*-2}\phi^2\,dx\), with an arbitrarily small increase in its
coefficient, and an error of order \(\varepsilon^{p^*}\). Lemma~\ref{lem:compact-embedding-high-p} gives
the convergence of these weighted integrals along weakly convergent sequences in \(\dot W^{1,p}\). The
error \(\varepsilon^{p^*}\) is \(o(\varepsilon^2)\).

To establish \eqref{eq:strategy-local-bound}, we argue by contradiction in
Section~\ref{sec:nonlinear-coercivity}. We rescale the perturbations as in
\eqref{eq:coercivity-normalization-low}, so that the integral in \eqref{eq:two-scale-normalization-low}
equals one. Lemma~\ref{lem:averaged-affine-lsc-low-p} shows that the weak limit belongs to \(\mathcal
H_U\). The weighted orthogonality passes to the limit, and Proposition~\ref{prop:affine-spectral-gap}
forces the weak limit to be zero. The estimates for the affine energy and the \(L^{p^*}(\mathbb
R^n)\)-norm, together with \eqref{eq:averaged-R-lower}, then imply that the integral in
\eqref{eq:two-scale-normalization-low} tends to zero. This contradicts the normalization.

Finally, if the global inequality failed, we could choose \(u_k\) such that
\[
  \|u_k\|_{L^{p^*}(\mathbb R^n)}=1, \qquad \frac{\mathcal E_{\rm aff,p}(u_k)}{S_{n,p}}-1 < \frac1k\,d_{\rm aff}(u_k,\mathcal M_{\rm aff})^2.
\]
Since \(d_{\rm aff}\le1\), the affine deficit tends to zero.
Proposition~\ref{prop:qualitative-affine-compactness-low-p} and Lemma~\ref{lem:affine-modulation-low-p}
allow us, after affine transformations and multiplication by constants, to write the functions as
\(U+\varepsilon_k\phi_k\), where
\[
  \varepsilon_k\to0,\qquad \|\nabla\phi_k\|_{L^p(\mathbb R^n)}=1,\qquad \phi_k\perp T_U\mathcal M_{\rm aff}.
\]
The definition of the affine distance gives \(d_{\rm aff}(u_k,\mathcal M_{\rm aff})\le C\varepsilon_k\).
By \eqref{eq:strategy-local-bound} and invariance of the normalized deficit,
\[
  \frac{\mathcal E_{\rm aff,p}(u_k)}{S_{n,p}}-1 \ge c\varepsilon_k^2 \ge c'\,d_{\rm aff}(u_k,\mathcal M_{\rm aff})^2.
\]
This contradicts the choice of \(u_k\) for large \(k\).

The paper is organized as follows. Section~\ref{sec:local-affine-expansion} derives the nonlinear lower
bounds for \(A_\xi(U+\varepsilon\varphi)\) and the identity \(Q_{\rm aff,p}=Q_{\rm Sob,p}-\mathcal
R_p\). Section~\ref{sec:nonlinear-coercivity} proves \(\ker_{\mathcal H_U}Q_{\rm aff,p}=T_U\mathcal
M_{\rm aff}\), establishes the corresponding spectral gap, and treats separately the ranges \(p^*\le2\)
and \(p^*>2\). Section~\ref{sec:proof-main-theorem} proves Theorem~\ref{thm:intro-main-affine-p<2}. The
case \(n=2\) for the kernel identity is treated in Appendix~\ref{app:kernel-affine-Hessian-n=2}.

\section{\texorpdfstring{Second variation and local estimates for \(1<p<2\)}{Second variation and local estimates for 1<p<2}}
\label{sec:local-affine-expansion}

Throughout this section, \(1<p<2\). Set
\[
  \delta_{\rm aff}(u)=\mathcal E_{\rm aff,p}(u)^p -S_{n,p}^p\|u\|_{L^{p^*}(\mathbb R^n)}^p.
\]
\subsection{The weighted Sobolev space}

Following \cite[Section~3.1]{FZ22}, set
\[
  C_{c,0}^{1}(\mathbb R^n) = \left\{ \phi\in C_c^1(\mathbb R^n): \phi \text{ is constant in a neighborhood of }0 \right\}.
\]
Let \(C_{c,0}^{\infty}(\mathbb R^n) =C_c^\infty(\mathbb R^n)\cap C_{c,0}^{1}(\mathbb R^n)\). We define
\begin{equation}
  \label{eq:def-HU-low}
  \mathcal H_U = \overline{C_{c,0}^{1}(\mathbb R^n)} ^{\|\cdot\|_{\mathcal H_U}}, \qquad \|\phi\|_{\mathcal H_U}^2 = \int_{\mathbb R^n} |\nabla U|^{p-2}|\nabla\phi|^2\,dx.
\end{equation}
The restriction at the origin is needed only for \(1<p\le(n+2)/(n+1)\). For \((n+2)/(n+1)<p<2\),
completing \(C_c^1(\mathbb R^n)\) in the norm in \eqref{eq:def-HU-low} gives the same space.

By \cite[Proposition~3.2]{FZ22},
\[
  \mathcal H_U \hookrightarrow\hookrightarrow L^2(\mathbb R^n;U^{p^*-2}dx).
\]
The weighted Poincar\'e inequality \cite[Lemma~3.3]{FZ22} gives \(c_0=c_0(n,p)>0\) such that
\begin{equation}
  \label{eq:weighted-Poincare-used-low}
  \int_{\mathbb R^n} U^{p^*-2}\phi^2\,dx \le c_0\|\phi\|_{\mathcal H_U}^2, \qquad \phi\in\mathcal H_U.
\end{equation}

The quadratic form associated with the linearized \(p\)-Laplacian satisfies
\begin{equation}
  \label{eq:principal-equivalence-gap-proof}
  (p-1)\|\phi\|_{\mathcal H_U}^2 \le \int_{\mathbb R^n}|\nabla U|^{p-2} \left( |\nabla\phi|^2 +(p-2)\frac{(\nabla U\cdot\nabla\phi)^2}{|\nabla U|^2} \right)\,dx \le \|\phi\|_{\mathcal H_U}^2.
\end{equation}

The tangent space to \(\mathcal M_{\rm aff}\) at \(U\) is
\[
  T_U\mathcal M_{\rm aff} = \operatorname{span} \left\{ U,\ Z_0,\ \partial_{x_1}U,\ldots,\partial_{x_n}U,\ x\cdot B\nabla U: B=B^T,\ \operatorname{tr}B=0 \right\}.
\]
Here \(Z_0=\frac{n-p}{p}U+x\cdot\nabla U\). Following \cite[Definition~3.7]{FZ22}, for \(\phi\in\dot
W^{1,p}(\mathbb R^n)\) we write
\begin{equation}
  \label{eq:orthogonality-convention}
  \phi\perp T_U\mathcal M_{\rm aff} \quad\Longleftrightarrow\quad \int_{\mathbb R^n}U^{p^*-2}Z\phi\,dx=0 \quad\text{for every }Z\in T_U\mathcal M_{\rm aff}.
\end{equation}
The pairings are finite by H\"older's inequality, since \(U^{p^*-2}Z\in L^{(p^*)'}(\mathbb R^n)\) for
each \(Z\in T_U\mathcal M_{\rm aff}\) and \(\phi\in L^{p^*}(\mathbb R^n)\) by the Sobolev inequality.

\subsection{\texorpdfstring{The quadratic form \(Q_{\rm aff,p}\)}{The quadratic form Q aff,p}}

Set \(\tau=n/p\) and define
\begin{equation}
  \label{eq:def-Phi-low-p}
  \Phi(a) = \left\langle a(\xi)^{-\tau} \right\rangle_{\mathbb S^{n-1}}^{-1/\tau}.
\end{equation}
For \(g\in L^2(\mathbb S^{n-1})\), write \(\operatorname{Var}_{\xi}(g) =\langle g^2\rangle_{\mathbb
S^{n-1}} -\langle g\rangle_{\mathbb S^{n-1}}^2\). Then \(\mathcal E_{\rm aff,p}(u)^p
=\alpha_{n,p}\Phi(A_\xi(u))\), where \(\alpha_{n,p}=m_{n,p}^{-1}\). Since \(U\) is radial,
\(A_\xi(U)=A_0\) for every \(\xi\in\mathbb S^{n-1}\). Moreover, \(-\Delta_pU=\Lambda U^{p^*-1}\), where
\(\Lambda=n((n-p)/(p-1))^{p-1}\).

For \(\phi\in\dot W^{1,p}(\mathbb R^n)\), set
\[
  L_\xi(\phi) = p\int_{\mathbb R^n} |\partial_\xi U|^{p-2} \partial_\xi U\,\partial_\xi\phi\,dx, \qquad h_t(\xi) = A_\xi(U+t\phi)-A_0.
\]
H\"older's inequality gives
\begin{equation}
  \label{eq:Lxi-Lp-bound}
  |L_\xi(\phi)| \le pA_0^{(p-1)/p}\|\partial_\xi\phi\|_{L^p(\mathbb R^n)} \le pA_0^{(p-1)/p}\|\nabla\phi\|_{L^p(\mathbb R^n)}.
\end{equation}
For the second variation, we take \(\phi\in C_{c,0}^\infty(\mathbb R^n)\).

The scalar inequality \(\bigl||a+b|^p-|a|^p-p|a|^{p-2}ab\bigr|\le C_p|b|^p\) gives
\[
  \left| h_t(\xi)-tL_\xi(\phi) \right| \le C_p|t|^p \int_{\mathbb R^n} |\partial_\xi\phi|^p\,dx \le C_p|t|^p \|\nabla\phi\|_{L^p(\mathbb R^n)}^p.
\]
Together with \eqref{eq:Lxi-Lp-bound}, this implies \(\|h_t\|_{L^\infty(\mathbb S^{n-1})}=O(|t|)\), and
\[
  \sup_{\xi\in\mathbb S^{n-1}} \left| \frac{h_t(\xi)}{t} - L_\xi(\phi) \right| \to0 \qquad \text{as }t\to0.
\]
In particular,
\begin{equation}
  \label{eq:affine-variance-limit-low-p}
  \operatorname{Var}_{\xi}(h_t) = t^2 \operatorname{Var}_{\xi}\bigl(L_\xi(\phi)\bigr) + o(t^2).
\end{equation}

The spherical mean of \(A_\xi(U+t\phi)\) satisfies
\[
  \left\langle A_\xi(U+t\phi) \right\rangle_{\mathbb S^{n-1}} = m_{n,p} \int_{\mathbb R^n} |\nabla U+t\nabla\phi|^p\,dx.
\]
Since \(|\nabla U|\) is bounded away from zero on \(\operatorname{supp}\nabla\phi\), Taylor's formula
gives
\begin{equation}
  \label{eq:affine-average-limit-low-p}
  \begin{aligned}
    \left\langle h_t \right\rangle_{\mathbb S^{n-1}} &= pm_{n,p}t \int_{\mathbb R^n} |\nabla U|^{p-2} \nabla U\cdot\nabla\phi\,dx \\
    &\quad+ \frac{pm_{n,p}}2t^2 \int_{\mathbb R^n} |\nabla U|^{p-2} \left( |\nabla\phi|^2 + (p-2) \frac{ (\nabla U\cdot\nabla\phi)^2 }{ |\nabla U|^2 } \right)\,dx + o(t^2).
  \end{aligned}
\end{equation}

For every \(\theta,\eta\in\mathbb R^n\) with \(|\theta|=1\),
\begin{equation}
  \label{eq:angular-identities-low-p}
  \begin{aligned}
    \left\langle |\theta\cdot\xi|^{p-2} (\theta\cdot\xi) (\eta\cdot\xi) \right\rangle_{\xi} &= m_{n,p}\,\theta\cdot\eta, \\
    \left\langle |\theta\cdot\xi|^{p-2} (\eta\cdot\xi)^2 \right\rangle_{\xi} &= \frac{m_{n,p}}{p-1} \left( |\eta|^2 + (p-2)(\theta\cdot\eta)^2 \right).
  \end{aligned}
\end{equation}
These identities follow by differentiation of \(\langle|\xi\cdot z|^p\rangle_\xi=m_{n,p}|z|^p\). By
\eqref{eq:angular-identities-low-p} and \(-\Delta_pU=\Lambda U^{p^*-1}\),
\begin{equation}
  \label{eq:average-Lxi-low-p}
  \alpha_{n,p} \left\langle L_\xi(\phi) \right\rangle_{\mathbb S^{n-1}} = p\Lambda \int_{\mathbb R^n} U^{p^*-1}\phi\,dx.
\end{equation}

Taylor's formula for \(\Phi\) at \(A_0\), together with \(\|h_t\|_{L^\infty(\mathbb S^{n-1})}=O(|t|)\),
yields
\begin{equation}
  \label{eq:Phi-expansion-low-p}
  \Phi(A_0+h_t) = A_0 + \left\langle h_t \right\rangle_{\mathbb S^{n-1}} - \frac{\tau+1}{2A_0} \operatorname{Var}_{\xi}(h_t) + o(t^2).
\end{equation}
By \eqref{eq:affine-variance-limit-low-p}, \eqref{eq:affine-average-limit-low-p}, and
\eqref{eq:Phi-expansion-low-p}, we obtain
\[
  \begin{aligned}
    \mathcal E_{\rm aff,p}(U+t\phi)^p &= \alpha_{n,p}A_0 + t\alpha_{n,p} \left\langle L_\xi(\phi) \right\rangle_{\mathbb S^{n-1}} \\
    &\quad+ \frac p2t^2 \int_{\mathbb R^n} |\nabla U|^{p-2} \left( |\nabla\phi|^2 + (p-2) \frac{ (\nabla U\cdot\nabla\phi)^2 }{ |\nabla U|^2 } \right)\,dx \\
    &\quad- t^2\alpha_{n,p} \frac{\tau+1}{2A_0} \operatorname{Var}_{\xi} \bigl(L_\xi(\phi)\bigr) + o(t^2).
  \end{aligned}
\]

For each fixed \(\phi\in C_{c,0}^\infty(\mathbb R^n)\), \eqref{eq:average-Lxi-low-p} and Taylor's
formula for \(S_{n,p}^p\|U+t\phi\|_{L^{p^*}(\mathbb R^n)}^p\) give
\begin{equation}
  \label{eq:affine-deficit-expansion-low-p}
  \mathcal E_{\rm aff,p}(U+t\phi)^p - S_{n,p}^p \|U+t\phi\|_{L^{p^*}(\mathbb R^n)}^p = \frac p2t^2 Q_{\rm aff,p}(\phi) + o(t^2),
\end{equation}
where \(Q_{\rm aff,p}(\phi)=Q_{\rm Sob,p}(\phi)-\mathcal R_p(\phi)\) and
\begin{equation}
  \label{sec2edit:def-Rp}
  \mathcal R_p(\phi) = \alpha_{n,p} \frac{\tau+1}{pA_0} \operatorname{Var}_{\xi} \bigl(L_\xi(\phi)\bigr).
\end{equation}
The Sobolev quadratic form is
\begin{equation}
  \label{eq:QSob-low-p}
  \begin{aligned}
    Q_{\rm Sob,p}(\phi) &= \int_{\mathbb R^n} |\nabla U|^{p-2} \left( |\nabla\phi|^2 + (p-2) \frac{ (\nabla U\cdot\nabla\phi)^2 }{ |\nabla U|^2 } \right)\,dx \\
    &\quad- (p^*-1)\Lambda \int_{\mathbb R^n} U^{p^*-2}\phi^2\,dx \\
    &\quad+ (p^*-p)\Lambda \left( \int_{\mathbb R^n} U^{p^*}\,dx \right)^{-1} \left( \int_{\mathbb R^n} U^{p^*-1}\phi\,dx \right)^2.
  \end{aligned}
\end{equation}

The Cauchy--Schwarz inequality with weight \(|\nabla U|^{p-2}\) gives
\begin{equation}
  \label{eq:Lxi-HU-bound}
  \begin{aligned}
    |L_\xi(\phi)|^2 &= p^2 \left| \int_{\mathbb R^n} |\partial_\xi U|^{p-2} \partial_\xi U\,\partial_\xi\phi\,dx \right|^2 \\
    &\le p^2 \left( \int_{\mathbb R^n} |\partial_\xi U|^{2p-2} |\nabla U|^{2-p}\,dx \right) \left( \int_{\mathbb R^n} |\nabla U|^{p-2} |\nabla\phi|^2\,dx \right) \\
    &\le p^2 \left( \int_{\mathbb R^n} |\nabla U|^p\,dx \right) \|\phi\|_{\mathcal H_U}^2.
  \end{aligned}
\end{equation}
Hence, by \eqref{eq:weighted-Poincare-used-low} and \eqref{eq:principal-equivalence-gap-proof}, \(Q_{\rm
aff,p}\) extends uniquely to a continuous quadratic form on \(\mathcal H_U\).

%For a general \(\varphi\in\dot W^{1,p}(\mathbb R^n)\), the weighted integral \(\int_{\mathbb R^n}|\nabla U|^{p-2}|\nabla\varphi|^2\,dx\) may be infinite.

\subsection{Nonlinear lower bounds}

\begin{lemma}
  \label{lem:affine-nonlinear-lower-p<2}
  Let \(1<p<2\) and \(0<\kappa<1\). There exist \(\varepsilon_0=\varepsilon_0(n,p,\kappa)>0\) and
  \(c_\kappa=c_\kappa(p,\kappa)>0\) such that, for \(0<\varepsilon\le\varepsilon_0\) and
  \(\varphi\in\dot W^{1,p}(\mathbb R^n)\) with \(\|\nabla\varphi\|_{L^p(\mathbb R^n)}=1\),
  \begin{equation}
    \label{eq:affine-lower-p<2}
    \mathcal E_{\rm aff,p}(U+\varepsilon\varphi)^p \ge \mathcal E_{\rm aff,p}(U)^p + \varepsilon\alpha_{n,p} \left\langle L_\xi(\varphi)
      \right\rangle_{\mathbb S^{n-1}} + \varepsilon^2 \mathcal N_{{\rm aff}}^{\varepsilon,\kappa}(\varphi) + c_\kappa\alpha_{n,p}
      \left\langle \mathfrak R_{\xi,\varepsilon}(\varphi) \right\rangle_{\mathbb S^{n-1}},
  \end{equation}
  where
  \begin{equation}
    \label{eq:def-N}
    \mathcal N_{{\rm aff}}^{\varepsilon,\kappa}(\varphi) ={} \alpha_{n,p}(1-\kappa) \left\langle \mathfrak B_{\xi,\varepsilon}(\varphi)
      \right\rangle_{\mathbb S^{n-1}} - \alpha_{n,p}(1+\kappa) \frac{\tau+1}{2A_0} \operatorname{Var}_{\xi} \bigl(L_\xi(\varphi)\bigr),
  \end{equation}
  and
  \begin{equation}
    \label{eq:def-B}
    \mathfrak B_{\xi,\varepsilon}(\varphi) = \frac{p}{2}\int_{\mathbb R^n} \Bigg[|\partial_\xi U|^{p-2} |\partial_\xi\varphi|^2 +
      (p-2)|w_{\xi,\varepsilon}|^{p-2} \left( \frac{ |\partial_\xi U| - |\partial_\xi U+\varepsilon\partial_\xi\varphi| }{\varepsilon}
      \right)^2 \Bigg]\,dx,
  \end{equation}
  and
  \begin{equation}
    \label{eq:def-R}
    \mathfrak R_{\xi,\varepsilon}(\varphi) = \int_{\mathbb R^n} \min \left\{ \varepsilon^p|\partial_\xi\varphi|^p,\, \varepsilon^2|\partial_\xi U|^{p-2} |\partial_\xi\varphi|^2 \right\}\,dx.
  \end{equation}
  The function \(w_{\xi,\varepsilon}\) is defined by
  \begin{equation}
    \label{eq:def-w}
    w_{\xi,\varepsilon} =
    \begin{cases}
      \displaystyle \left( \frac{ |\partial_\xi U+\varepsilon\partial_\xi\varphi| }{ (2-p)|\partial_\xi U+\varepsilon\partial_\xi\varphi| +
        (p-1)|\partial_\xi U| } \right)^{\frac1{p-2}} \partial_\xi U, & |\partial_\xi U| < |\partial_\xi U+\varepsilon\partial_\xi\varphi|,
        \\[1.2em]
      \partial_\xi U, & |\partial_\xi U+\varepsilon\partial_\xi\varphi| \le |\partial_\xi U|.
    \end{cases}
  \end{equation}
\end{lemma}

\begin{proof}
  Fix \(0<\kappa_*<1\). By the reverse triangle inequality and \eqref{eq:def-w},
  \begin{equation}
    \label{eq:Bge0}
    \begin{aligned}
      \frac p2 |\partial_\xi U|^{p-2}|\partial_\xi\varphi|^2 &+ \frac{p(p-2)}2 |w_{\xi,\varepsilon}|^{p-2} \left( \frac{ |\partial_\xi
        U+\varepsilon\partial_\xi\varphi| - |\partial_\xi U| }{ \varepsilon } \right)^2 \\
      &\qquad\ge \frac p2 \left( |\partial_\xi U|^{p-2} - (2-p)|w_{\xi,\varepsilon}|^{p-2} \right) |\partial_\xi\varphi|^2 \ge0.
    \end{aligned}
  \end{equation}
  Indeed, if \(|\partial_\xi U+\varepsilon\partial_\xi\varphi| \le|\partial_\xi U|\), then
  \(w_{\xi,\varepsilon}=\partial_\xi U\). Otherwise,
  \[
    |\partial_\xi U|^{p-2} - (2-p)|w_{\xi,\varepsilon}|^{p-2} = \frac{ (p-1)|\partial_\xi U| }{ (2-p)|\partial_\xi U+\varepsilon\partial_\xi\varphi| + (p-1)|\partial_\xi U| } |\partial_\xi U|^{p-2} \ge0.
  \]
  By \cite[Lemma~2.1(i)]{FZ22} and the \((p-1)\)-H\"older continuity of \(z\mapsto|z|^{p-2}z\),
  \[
    \begin{aligned}
      0 &\le \frac{\varepsilon^2}{2} \Bigg[ p|\partial_\xi U|^{p-2} |\partial_\xi\varphi|^2 + p(p-2)|w_{\xi,\varepsilon}|^{p-2} \left(
        \frac{ |\partial_\xi U| - |\partial_\xi U+\varepsilon\partial_\xi\varphi| }{\varepsilon} \right)^2 \Bigg] \\
      &\le \frac{1}{1-\kappa_*} \Bigl( |\partial_\xi U+\varepsilon\partial_\xi\varphi|^p - |\partial_\xi U|^p - p|\partial_\xi U|^{p-2} \partial_\xi U\, \varepsilon\partial_\xi\varphi \Bigr) \\
      &= \frac{p}{1-\kappa_*} \int_0^1 \Bigl( |\partial_\xi U+t\varepsilon\partial_\xi\varphi|^{p-2} \bigl( \partial_\xi
        U+t\varepsilon\partial_\xi\varphi \bigr) - |\partial_\xi U|^{p-2}\partial_\xi U \Bigr) \varepsilon\partial_\xi\varphi\,dt \\
      &\le C_{p,\kappa_*} \varepsilon^p|\partial_\xi\varphi|^p.
    \end{aligned}
  \]
  Integrating and using \eqref{eq:def-R}, we obtain
  \begin{equation}
    \label{eq:B-R-upper-p<2}
    0\le\varepsilon^2\mathfrak B_{\xi,\varepsilon}(\varphi) \le C_{p,\kappa_*}\varepsilon^p A_\xi(\varphi) \le
      C_{p,\kappa_*}\varepsilon^p,\quad 0\le\mathfrak R_{\xi,\varepsilon}(\varphi) \le\varepsilon^p A_\xi(\varphi)\le\varepsilon^p.
  \end{equation}
  Thus by \cite[Lemma~2.1(i)]{FZ22}, with \(x=\partial_\xi U\) and \(y=\varepsilon\partial_\xi\varphi\),
  gives
  \begin{equation}
    \label{eq:directional-FZ-used-p<2}
    \begin{aligned}
      A_\xi(U+\varepsilon\varphi) &\ge A_0 + \varepsilon L_\xi(\varphi) + (1-\kappa_*)\varepsilon^2 \mathfrak B_{\xi,\varepsilon}(\varphi) + c_{\kappa_*} \mathfrak R_{\xi,\varepsilon}(\varphi).
    \end{aligned}
  \end{equation}

  Set
  \[
    h_\varepsilon(\xi) = \varepsilon L_\xi(\varphi)+s_\varepsilon(\xi), \qquad s_\varepsilon(\xi) = (1-\kappa_*)\varepsilon^2 \mathfrak
      B_{\xi,\varepsilon}(\varphi) + c_{\kappa_*} \mathfrak R_{\xi,\varepsilon}(\varphi).
  \]
  By \eqref{eq:Lxi-Lp-bound} and \eqref{eq:B-R-upper-p<2}, \(0\le s_\varepsilon(\xi)\le
  C_{p,\kappa_*}\varepsilon^p\) and \(\|h_\varepsilon\|_{L^\infty(\mathbb S^{n-1})}\le C\varepsilon\).
  By monotonicity of \(\Phi\) and \eqref{eq:directional-FZ-used-p<2}, \(\mathcal E_{\rm
  aff,p}(U+\varepsilon\varphi)^p \ge\alpha_{n,p}\Phi(A_0+h_\varepsilon)\).

  For each \(\gamma>0\), there is \(\delta_0>0\) such that
  \begin{equation}
    \label{eq:Phi-lower-general-p<2}
    \Phi(A_0+h) \ge A_0 + \left\langle h\right\rangle_{\mathbb S^{n-1}} - (1+\gamma) \frac{\tau+1}{2A_0} \operatorname{Var}_{\xi}(h)
  \end{equation}
  whenever \(\|h\|_{L^\infty(\mathbb S^{n-1})}\le\delta_0\). By Young's inequality and \(0\le
  s_\varepsilon\le C_{p,\kappa_*}\varepsilon^p\) give
  \begin{equation}
    \label{eq:bounded-var}
    \begin{aligned}
      \operatorname{Var}_{\xi}(h_\varepsilon) &\le (1+\gamma)\varepsilon^2 \operatorname{Var}_{\xi} \bigl(L_\xi(\varphi)\bigr) + C_\gamma \operatorname{Var}_{\xi}(s_\varepsilon), \\
      \operatorname{Var}_{\xi}(s_\varepsilon) &\le \left\langle s_\varepsilon^2 \right\rangle_{\mathbb S^{n-1}} \le C_{p,\kappa_*}\varepsilon^p \left\langle s_\varepsilon \right\rangle_{\mathbb S^{n-1}}.
    \end{aligned}
  \end{equation}
  By \eqref{eq:Phi-lower-general-p<2}, for sufficiently small \(\varepsilon_0\) with
  \(C_\gamma\varepsilon_0^p\le\gamma\) and \eqref{eq:bounded-var},
  \[
    \begin{aligned}
      \Phi(A_0+h_\varepsilon) &\ge A_0 + \varepsilon \left\langle L_\xi(\varphi) \right\rangle_{\mathbb S^{n-1}} + \left\langle s_\varepsilon \right\rangle_{\mathbb S^{n-1}} \nonumber \\
      &\quad - (1+\gamma)^2 \frac{\tau+1}{2A_0} \varepsilon^2 \operatorname{Var}_{\xi} \bigl(L_\xi(\varphi)\bigr) - C_\gamma\varepsilon^p \left\langle s_\varepsilon \right\rangle_{\mathbb S^{n-1}} \nonumber \\
      &\ge A_0 + \varepsilon \left\langle L_\xi(\varphi) \right\rangle_{\mathbb S^{n-1}} + (1-\gamma) \left\langle s_\varepsilon \right\rangle_{\mathbb S^{n-1}} \nonumber \\
      &\quad - (1+\gamma)^2 \frac{\tau+1}{2A_0} \varepsilon^2 \operatorname{Var}_{\xi} \bigl(L_\xi(\varphi)\bigr) \nonumber \\
      &= A_0 + \varepsilon \left\langle L_\xi(\varphi) \right\rangle_{\mathbb S^{n-1}} + (1-\gamma)(1-\kappa_*)\varepsilon^2 \left\langle
        \mathfrak B_{\xi,\varepsilon}(\varphi) \right\rangle_{\mathbb S^{n-1}} \nonumber \\
      &\quad - (1+\gamma)^2 \frac{\tau+1}{2A_0} \varepsilon^2 \operatorname{Var}_{\xi} \bigl(L_\xi(\varphi)\bigr) + (1-\gamma)c_{\kappa_*}
        \left\langle \mathfrak R_{\xi,\varepsilon}(\varphi) \right\rangle_{\mathbb S^{n-1}}.
    \end{aligned}
  \]
  Taking \(\gamma=\kappa_*=\kappa/4\) and \(c_\kappa=(1-\gamma)c_{\kappa_*}\), we obtain
  \eqref{eq:affine-lower-p<2} from \eqref{eq:def-N} and \(\alpha_{n,p}A_0=\mathcal E_{\rm aff,p}(U)^p\).
\end{proof}

Assume \(\|\nabla\varphi\|_{L^p(\mathbb R^n)}=1\) and \(0<\varepsilon\le\varepsilon_0\).

Suppose first that \(1<p\le2n/(n+2)\), or equivalently \(p^*\le2\). Define
\begin{equation}
  \label{sec2edit:def-Pepsilon}
  \mathcal P_\varepsilon(\varphi) = \int_{\mathbb R^n} \frac{ \bigl(U+C_1|\varepsilon\varphi|\bigr)^{p^*} }{ U^2+|\varepsilon\varphi|^2 } |\varphi|^2\,dx,
\end{equation}
where \(C_1=C_1(p^*,\kappa)>0\) is chosen as in \cite[Lemma~2.4(i)]{FZ22}, for fixed \(0<\kappa<1\).
Applying that lemma gives
\begin{equation}
  \label{eq:Lpstar-upper-low-p}
  \begin{aligned}
    S_{n,p}^p \|U+\varepsilon\varphi\|_{L^{p^*}(\mathbb R^n)}^p &\le S_{n,p}^p \|U\|_{L^{p^*}(\mathbb R^n)}^p + \varepsilon p\Lambda \int_{\mathbb R^n} U^{p^*-1}\varphi\,dx \\
    &\quad+ \varepsilon^2\Lambda \left( \frac{p(p^*-1)}2 + \frac{p\kappa}{p^*} \right) \mathcal P_\varepsilon(\varphi).
  \end{aligned}
\end{equation}
By \eqref{eq:average-Lxi-low-p}, \eqref{eq:affine-lower-p<2}, and \eqref{eq:Lpstar-upper-low-p}, we
obtain
\begin{equation}
  \label{eq:deficit-lower-low-p-before-prop}
  \delta_{\rm aff}(U+\varepsilon\varphi) \ge \varepsilon^2 \left[ \mathcal N_{{\rm aff}}^{\varepsilon,\kappa}(\varphi) - \Lambda \left(
    \frac{p(p^*-1)}2 + \frac{p\kappa}{p^*} \right) \mathcal P_\varepsilon(\varphi) \right] + c_\kappa\alpha_{n,p} \left\langle \mathfrak
    R_{\xi,\varepsilon}(\varphi) \right\rangle_{\mathbb S^{n-1}}.
\end{equation}

The following proposition is proved in Section~\ref{sec:nonlinear-coercivity}.

\begin{proposition}
  \label{prop:coercivity-low-p}
  Let \(1<p\le2n/(n+2)\). There exist \(\kappa_0=\kappa_0(n,p)\in(0,1)\) and \(\eta_0=\eta_0(n,p)>0\)
  such that, for each \(0<\kappa\le\kappa_0\), one can choose
  \(\varepsilon_0=\varepsilon_0(n,p,\kappa)>0\) for which
  \begin{equation}
    \label{eq:coercivity-low-p}
    \begin{aligned}
      &\mathcal N_{{\rm aff}}^{\varepsilon,\kappa}(\varphi) + \frac{c_\kappa}{2} \alpha_{n,p}\varepsilon^{-2} \left\langle \mathfrak
        R_{\xi,\varepsilon}(\varphi) \right\rangle_{\mathbb S^{n-1}} \ge \frac p2 \bigl( (p^*-1)\Lambda+\eta_0 \bigr) \mathcal
        P_\varepsilon(\varphi).
    \end{aligned}
  \end{equation}
  whenever \(0<\varepsilon\le\varepsilon_0\) and \(\varphi\in\dot W^{1,p}(\mathbb R^n)\) satisfy
  \(\|\nabla\varphi\|_{L^p(\mathbb R^n)}=1\) and \(\varphi\perp T_U\mathcal M_{\rm aff}\).
\end{proposition}

For \(\varphi\perp T_U\mathcal M_{\rm aff}\), choose \(0<\kappa\le\kappa_0\) so that \(\frac
p2\eta_0-\frac{p\kappa}{p^*}\Lambda>0\). Applying \eqref{eq:coercivity-low-p} to
\eqref{eq:deficit-lower-low-p-before-prop} gives
\begin{equation}
  \label{eq:deficit-after-low-coercivity}
  \begin{aligned}
    \delta_{\rm aff}(U+\varepsilon\varphi) &\ge \varepsilon^2 \left( \frac p2\eta_0 - \frac{p\kappa}{p^*}\Lambda \right) \mathcal
      P_\varepsilon(\varphi) + \frac{c_\kappa}{2} \alpha_{n,p} \left\langle \mathfrak R_{\xi,\varepsilon}(\varphi) \right\rangle_{\mathbb
      S^{n-1}} \\
    &\ge c\varepsilon^2\mathcal P_\varepsilon(\varphi) + \frac{c_\kappa}{2} \alpha_{n,p} \left\langle \mathfrak R_{\xi,\varepsilon}(\varphi) \right\rangle_{\mathbb S^{n-1}}.
  \end{aligned}
\end{equation}

For every \(1<p<2\), there is \(c_p>0\) such that
\begin{equation}
  \label{eq:minimum-remainder-coercivity}
  \min \left\{ t^p,\, a^{p-2}t^2 \right\} \ge c_p(a+t)^{p-2}t^2 \qquad \text{for every }a,t\ge0.
\end{equation}
Set \(a=|\partial_\xi U|\) and \(t=\varepsilon|\partial_\xi\varphi|\) in
\eqref{eq:minimum-remainder-coercivity}. Since \(p-2<0\),
\[
  \begin{aligned}
    \mathfrak R_{\xi,\varepsilon}(\varphi) &\ge c_p\varepsilon^2 \int_{\mathbb R^n} \bigl( |\partial_\xi U| + \varepsilon|\partial_\xi\varphi| \bigr)^{p-2} |\partial_\xi\varphi|^2\,dx \\
    &\ge c_p\varepsilon^2 \int_{\mathbb R^n} \bigl( |\nabla U| + \varepsilon|\nabla\varphi| \bigr)^{p-2} |\nabla\varphi\cdot\xi|^2\,dx.
  \end{aligned}
\]
Average over \(\mathbb S^{n-1}\) and use \(\langle|\nabla\varphi\cdot\xi|^2\rangle_\xi
=|\nabla\varphi|^2/n\). Then
\begin{equation}
  \label{eq:averaged-R-lower}
  \varepsilon^{-2} \left\langle \mathfrak R_{\xi,\varepsilon}(\varphi) \right\rangle_{\mathbb S^{n-1}} \ge \frac{c_p}{n} \int_{\mathbb R^n}
    \bigl( |\nabla U| + \varepsilon|\nabla\varphi| \bigr)^{p-2} |\nabla\varphi|^2\,dx.
\end{equation}
H\"older's inequality gives
\begin{equation}
  \label{eq:two-scale-energy-lower}
  \begin{aligned}
    1 &= \int_{\mathbb R^n} |\nabla\varphi|^p\,dx \\
    &\le \left( \int_{\mathbb R^n} \bigl( |\nabla U| + \varepsilon|\nabla\varphi| \bigr)^{p-2} |\nabla\varphi|^2\,dx \right)^{p/2} \left(
      \int_{\mathbb R^n} \bigl( |\nabla U| + \varepsilon|\nabla\varphi| \bigr)^p\,dx \right)^{(2-p)/2}.
  \end{aligned}
\end{equation}
For \(0<\varepsilon\le1\),
\[
  \begin{aligned}
    \int_{\mathbb R^n} \bigl( |\nabla U| + \varepsilon|\nabla\varphi| \bigr)^p\,dx &\le C_p \int_{\mathbb R^n} \left( |\nabla U|^p + \varepsilon^p|\nabla\varphi|^p \right)\,dx \le C_{n,p}.
  \end{aligned}
\]
Substitution into \eqref{eq:two-scale-energy-lower} yields
\begin{equation}
  \label{eq:two-scale-energy-positive}
  \int_{\mathbb R^n} \bigl( |\nabla U| + \varepsilon|\nabla\varphi| \bigr)^{p-2} |\nabla\varphi|^2\,dx \ge c_{n,p}.
\end{equation}
By \eqref{eq:averaged-R-lower} and \eqref{eq:two-scale-energy-positive} imply
\begin{equation}
  \label{eq:R-controls-epsilon2}
  \left\langle \mathfrak R_{\xi,\varepsilon}(\varphi) \right\rangle_{\mathbb S^{n-1}} \ge c_{n,p}\varepsilon^2.
\end{equation}
For \(1<p\le2n/(n+2)\), \eqref{eq:deficit-after-low-coercivity} and \eqref{eq:R-controls-epsilon2} give
\begin{equation}
  \label{eq:local-lower-low-p}
  \delta_{\rm aff}(U+\varepsilon\varphi) \ge c\varepsilon^2.
\end{equation}

Suppose next that \(2n/(n+2)<p<2\), so \(p^*>2\). For every \(\kappa>0\), \cite[Lemma~2.4(ii)]{FZ22} and
the concavity of \(s\mapsto s^{p/p^*}\) give
\begin{equation}
  \label{eq:Lpstar-upper-high-p}
  \begin{aligned}
    S_{n,p}^p \|U+\varepsilon\varphi\|_{L^{p^*}(\mathbb R^n)}^p &\le S_{n,p}^p \|U\|_{L^{p^*}(\mathbb R^n)}^p + \varepsilon p\Lambda
      \int_{\mathbb R^n} U^{p^*-1}\varphi\,dx+ C_\kappa\varepsilon^{p^*} \int_{\mathbb R^n} |\varphi|^{p^*}\,dx \\
    &\quad+ \varepsilon^2\Lambda \left( \frac{p(p^*-1)}2+\kappa \right) \int_{\mathbb R^n} U^{p^*-2}\varphi^2\,dx .
  \end{aligned}
\end{equation}

By \eqref{eq:average-Lxi-low-p}, \eqref{eq:affine-lower-p<2}, and \eqref{eq:Lpstar-upper-high-p}, for
\(0<\kappa<1\),
\begin{equation}
  \label{eq:deficit-lower-high-p-before-prop}
  \begin{aligned}
    \delta_{\rm aff}(U+\varepsilon\varphi) &\ge \varepsilon^2 \left[ \mathcal N_{{\rm aff}}^{\varepsilon,\kappa}(\varphi) - \Lambda \left(
      \frac{p(p^*-1)}2+\kappa \right) \int_{\mathbb R^n} U^{p^*-2}\varphi^2\,dx \right] \\
    &\quad+ c_\kappa\alpha_{n,p} \left\langle \mathfrak R_{\xi,\varepsilon}(\varphi) \right\rangle_{\mathbb S^{n-1}} - C_\kappa\varepsilon^{p^*} \int_{\mathbb R^n} |\varphi|^{p^*}\,dx .
  \end{aligned}
\end{equation}

The corresponding coercivity estimate for \(p^*>2\) is also proved in
Section~\ref{sec:nonlinear-coercivity}.

\begin{proposition}
  \label{prop:coercivity-high-p}
  Let \(2n/(n+2)<p<2\). There exist \(\kappa_0=\kappa_0(n,p)\in(0,1)\) and \(\eta_0=\eta_0(n,p)>0\) such
  that, for each \(0<\kappa\le\kappa_0\), one can choose \(\varepsilon_0=\varepsilon_0(n,p,\kappa)>0\)
  for which
  \begin{equation}
    \label{eq:coercivity-high-p}
    \begin{aligned}
      &\mathcal N_{{\rm aff}}^{\varepsilon,\kappa}(\varphi) + \frac{c_\kappa}{2} \alpha_{n,p}\varepsilon^{-2} \left\langle \mathfrak
        R_{\xi,\varepsilon}(\varphi) \right\rangle_{\mathbb S^{n-1}} \ge \frac p2 \bigl( (p^*-1)\Lambda+\eta_0 \bigr) \int_{\mathbb R^n}
        U^{p^*-2}\varphi^2\,dx.
    \end{aligned}
  \end{equation}
  whenever \(0<\varepsilon\le\varepsilon_0\) and \(\varphi\in\dot W^{1,p}(\mathbb R^n)\) satisfy
  \(\|\nabla\varphi\|_{L^p(\mathbb R^n)}=1\) and \(\varphi\perp T_U\mathcal M_{\rm aff}\).
\end{proposition}

Assume \(\varphi\perp T_U\mathcal M_{\rm aff}\), and choose \(0<\kappa\le\kappa_0\) so that \(\frac
p2\eta_0-\kappa\Lambda>0\). Then \eqref{eq:coercivity-high-p} and
\eqref{eq:deficit-lower-high-p-before-prop} give
\begin{equation}
  \label{eq:deficit-after-high-coercivity}
  \begin{aligned}
    \delta_{\rm aff}(U+\varepsilon\varphi) \ge c\varepsilon^2 \int_{\mathbb R^n} U^{p^*-2}\varphi^2\,dx + \frac{c_\kappa}{2} \alpha_{n,p}
      \left\langle \mathfrak R_{\xi,\varepsilon}(\varphi) \right\rangle_{\mathbb S^{n-1}} - C_\kappa\varepsilon^{p^*} \int_{\mathbb R^n}
      |\varphi|^{p^*}\,dx .
  \end{aligned}
\end{equation}
Since \(p^*>2\), the Sobolev inequality and \eqref{eq:R-controls-epsilon2} give, for sufficiently small
\(\varepsilon_0\),
\begin{equation}
  \label{eq:absorb-Lpstar-remainder-high-p}
  \begin{aligned}
    C_\kappa\varepsilon^{p^*} \int_{\mathbb R^n} |\varphi|^{p^*}\,dx &\le C_{n,p,\kappa}\varepsilon^{p^*} \le \frac{c_\kappa}{4}
      \alpha_{n,p} \left\langle \mathfrak R_{\xi,\varepsilon}(\varphi) \right\rangle_{\mathbb S^{n-1}}.
  \end{aligned}
\end{equation}
Combining \eqref{eq:deficit-after-high-coercivity}, \eqref{eq:absorb-Lpstar-remainder-high-p}, and
\eqref{eq:R-controls-epsilon2}, we conclude that
\begin{equation}
  \label{eq:local-lower-high-p}
  \begin{aligned}
    \delta_{\rm aff}(U+\varepsilon\varphi) &\ge c\varepsilon^2 \int_{\mathbb R^n} U^{p^*-2}\varphi^2\,dx + \frac{c_\kappa}{4} \alpha_{n,p}
      \left\langle \mathfrak R_{\xi,\varepsilon}(\varphi) \right\rangle_{\mathbb S^{n-1}} \ge c\varepsilon^2.
  \end{aligned}
\end{equation}

\section{Nonlinear coercivity near the Talenti extremal}
\label{sec:nonlinear-coercivity}

\subsection{Spectral gap and compactness}

We begin with the spectral gap for \(Q_{\rm aff,p}\) and the lower semicontinuity estimate
\eqref{eq:B-lower-cts}.

\begin{proposition}
  \label{prop:affine-kernel-n-ge3-low-p}
  Let \(n\ge3\) and \(1<p<2\). Then \(\ker_{\mathcal H_U}Q_{\rm aff,p}=T_U\mathcal M_{\rm aff}\).
\end{proposition}
\begin{proof}
  We follow \cite[Section~5]{FLZ26B}, working first on \(C_{c,0}^{\infty}(\mathbb R^n)\) and then
  extending to \(\mathcal H_U\) by density and continuity. For \(1<p<2\), direct computation shows that
  the coefficients \(\nu_m\) defined in \cite[Subsection~5.7]{FLZ26B} satisfy
  \[
    \nu_1=1, \qquad 0<\nu_{m+1}<\nu_m \quad\text{for every }m\ge1.
  \]
  Thus \(\nu_m<1\) for \(m\ge2\), and the argument of \cite[Subsection~5.8]{FLZ26B} gives
  \[
    \ker_{\mathcal H_U}Q_{\rm aff,p} =T_U\mathcal M_{\rm aff}.
  \]
  This completes the proof.
\end{proof}

The kernel description is therefore available in every dimension: for \(n\ge3\) by
Proposition~\ref{prop:affine-kernel-n-ge3-low-p}, and for \(n=2\) by
Proposition~\ref{prop:affine-kernel-n=2}. We now upgrade this nondegeneracy to a uniform coercive gap.

\begin{proposition}
  \label{prop:affine-spectral-gap}
  Let \(1<p<2\). There exists \(c_{\rm sg}=c_{\rm sg}(n,p)>0\) such that \(Q_{\rm aff,p}(\phi)\ge c_{\rm
  sg}\|\phi\|_{\mathcal H_U}^2\) for every \(\phi\in\mathcal H_U\) with \(\phi\perp T_U\mathcal M_{\rm
  aff}\).
\end{proposition}

\begin{proof}
  For \(n\ge3\), the kernel computation in \cite[Section~5]{FLZ26B} remains valid for \(1<p<2\), while
  the case \(n=2\) is established in Proposition~\ref{prop:affine-kernel-n=2}. Hence
  \[
    \ker_{\mathcal H_U}Q_{\rm aff,p} = T_U\mathcal M_{\rm aff}.
  \]
  The spectral gap now follows by the same compactness argument as in \cite[Proposition~3.1]{FLZ26B}. We
  omit the proof.
\end{proof}

\begin{lemma}
  \label{lem:compact-embedding-high-p}
  If \(\frac{2n}{n+2}<p<n\), then
  \begin{equation}
    \label{eq:compact-embedding-high-p}
    \dot W^{1,p}(\mathbb R^n) \hookrightarrow\hookrightarrow L^2(\mathbb R^n;U^{p^*-2}dx).
  \end{equation}
\end{lemma}

\begin{proof}
  Let \(\phi_j\rightharpoonup\phi\) in \(\dot W^{1,p}(\mathbb R^n)\). By the Sobolev inequality,
  \((\phi_j)\) is bounded in \(L^{p^*}(\mathbb R^n)\). Since \(p^*>2\), the Rellich--Kondrachov theorem
  gives a subsequence such that
  \begin{equation}
    \label{eq:local-L2-convergence-high-p}
    \phi_j\to\phi \qquad \text{strongly in }L^2(B_R)
  \end{equation}
  for every \(R>0\).

  By H\"older's inequality and the Sobolev inequality,
  \begin{equation}
    \label{eq:weighted-tail-high-p}
    \begin{aligned}
      \int_{\mathbb R^n\setminus B_R} U^{p^*-2}|\phi_j-\phi|^2\,dx &\le \left( \int_{\mathbb R^n\setminus B_R} U^{p^*}\,dx
        \right)^{1-\frac{2}{p^*}} \left( \int_{\mathbb R^n} |\phi_j-\phi|^{p^*}\,dx \right)^{\frac{2}{p^*}} \\
      &\le C \left( \int_{\mathbb R^n\setminus B_R} U^{p^*}\,dx \right)^{1-\frac{2}{p^*}},
    \end{aligned}
  \end{equation}
  with \(C\) independent of \(j\). Since \(U\in L^{p^*}(\mathbb R^n)\) and \(U^{p^*-2}\) is bounded on
  \(B_R\), \eqref{eq:local-L2-convergence-high-p} and \eqref{eq:weighted-tail-high-p} imply
  \(\phi_j\to\phi\) strongly in \(L^2(\mathbb R^n;U^{p^*-2}dx)\).
\end{proof}

\begin{lemma}
  \label{lem:averaged-affine-lsc-low-p}
  Let \(1<p<2\). Assume that \(\delta_j\downarrow0\), \(\psi_j\rightharpoonup\psi\) in \(\dot
  W^{1,p}(\mathbb R^n)\), and
  \[
    \sup_j \int_{\mathbb R^n} \bigl( |\nabla U|+\delta_j|\nabla\psi_j| \bigr)^{p-2} |\nabla\psi_j|^2\,dx <\infty.
  \]
  Then \(\psi\in\mathcal H_U\). Moreover,
  \begin{equation}
    \label{eq:B-lower-cts}
    \liminf_{j\to\infty} \alpha_{n,p} \left\langle \mathfrak B_{\xi,\delta_j}(\psi_j) \right\rangle_{\mathbb S^{n-1}} \ge \frac p2
      \int_{\mathbb R^n} |\nabla U|^{p-2} \left( |\nabla\psi|^2 + (p-2) \frac{(\nabla U\cdot\nabla\psi)^2} {|\nabla U|^2} \right)\,dx.
  \end{equation}
\end{lemma}

\begin{proof}
  Fix \(R>1\), \(0<\sigma<1\), and \(0<\eta<1/2\). Set \(A_R=B_R\setminus\overline{B_{1/R}}\) and define
  \[
    \begin{aligned}
      \Gamma_\sigma(x) &= \left\{ \xi\in\mathbb S^{n-1}: |\partial_\xi U(x)| \ge \sigma|\nabla U(x)| \right\}, \\
      G_{j,R}^{\eta,\sigma} &= \left\{ x\in A_R: \delta_j|\nabla\psi_j(x)| \le \eta\sigma|\nabla U(x)| \right\}.
    \end{aligned}
  \]

  Since \(|\nabla U|\) is bounded away from zero on \(A_R\) and \((\nabla\psi_j)\) is bounded in
  \(L^p(\mathbb R^n)\), Chebyshev's inequality gives
  \begin{equation}
    \label{eq:measure-limit}
    |A_R\setminus G_{j,R}^{\eta,\sigma}| \to0.
  \end{equation}
  On \(G_{j,R}^{\eta,\sigma}\), we have \(|\nabla U|+\delta_j|\nabla\psi_j| \le(1+\eta\sigma)|\nabla
  U|\). Since \(p<2\),
  \[
    \begin{aligned}
      \int_{G_{j,R}^{\eta,\sigma}} |\nabla U|^{p-2}|\nabla\psi_j|^2\,dx &\le (1+\eta\sigma)^{2-p} \int_{\mathbb R^n} \bigl( |\nabla U|+\delta_j|\nabla\psi_j| \bigr)^{p-2} |\nabla\psi_j|^2\,dx \le C.
    \end{aligned}
  \]
  The weight \(|\nabla U|^{p-2}\) is bounded above and below by positive constants on \(A_R\). Hence
  \(\mathbf 1_{G_{j,R}^{\eta,\sigma}}\nabla\psi_j\) is bounded in \(L^2(A_R;\mathbb R^n)\). In fact,
  \begin{equation}
    \label{eq:weak-convergence-L2}
    \mathbf 1_{G_{j,R}^{\eta,\sigma}}\nabla\psi_j \rightharpoonup \nabla\psi \qquad \text{weakly in }L^2(A_R;\mathbb R^n).
  \end{equation}
  For \(\zeta\in C_c^\infty(A_R;\mathbb R^n)\), H\"older's inequality and \eqref{eq:measure-limit} give
  \[
    \begin{aligned}
      \left| \int_{A_R\setminus G_{j,R}^{\eta,\sigma}} \nabla\psi_j\cdot\zeta\,dx \right| &\le \|\nabla\psi_j\|_{L^p(\mathbb R^n)} \|\zeta\|_{L^\infty(A_R)} |A_R\setminus G_{j,R}^{\eta,\sigma}|^{1/p'} \to0.
    \end{aligned}
  \]
  Since \(\nabla\psi_j\rightharpoonup\nabla\psi\) in \(L^p(\mathbb R^n)\), every weak \(L^2\)-limit is
  \(\nabla\psi\). Thus \eqref{eq:weak-convergence-L2} holds, and weak lower semicontinuity yields
  \begin{equation}
    \label{eq:weighted-limit-bound-annulus}
    \begin{aligned}
      \int_{A_R} |\nabla U|^{p-2}|\nabla\psi|^2\,dx &\le \liminf_{j\to\infty} \int_{G_{j,R}^{\eta,\sigma}} |\nabla U|^{p-2}|\nabla\psi_j|^2\,dx \le C.
    \end{aligned}
  \end{equation}
  Thus \(\psi\in W_{\rm loc}^{1,2}(\mathbb R^n\setminus\{0\})\). The constant in
  \eqref{eq:weighted-limit-bound-annulus} is independent of \(R\). Monotone convergence as
  \(R\to\infty\) gives
  \begin{equation}
    \label{eq:psi-weight-norm}
    \int_{\mathbb R^n} |\nabla U|^{p-2}|\nabla\psi|^2\,dx <\infty.
  \end{equation}

  We approximate \(\psi\) by functions in \(C_{c,0}^1(\mathbb R^n)\) in the norm \eqref{eq:def-HU-low}.
  Recall that \(\psi\in\dot W^{1,p}(\mathbb R^n)\), and
  \[
    |\nabla U(x)|^{p-2} \asymp
    \begin{cases}
      |x|^{\frac{p-2}{p-1}}, & |x|\le1, \\[1mm]
      |x|^{\frac{(n-1)(2-p)}{p-1}}, & |x|\ge1.
    \end{cases}
  \]

  For \(0<\rho<1/2\), set \(c_\rho = \frac{1}{|B_{2\rho}\setminus B_\rho|} \int_{B_{2\rho}\setminus
  B_\rho}\psi\,dx\). Choose \(\eta_\rho\in C^\infty(\mathbb R^n)\) such that \(\eta_\rho=0\) on
  \(B_\rho\), \(\eta_\rho=1\) on \(\mathbb R^n\setminus B_{2\rho}\), and \(|\nabla\eta_\rho|\le
  C/\rho\). Set \(\psi_\rho=c_\rho+\eta_\rho(\psi-c_\rho)\). Since \(|\nabla
  U|^{p-2}\asymp\rho^{(p-2)/(p-1)}\) on \(B_{2\rho}\setminus B_\rho\), the scaled Poincar\'e inequality
  gives
  \[
    \frac1{\rho^2} \int_{B_{2\rho}\setminus B_\rho} |\nabla U|^{p-2}|\psi-c_\rho|^2\,dx \le C \int_{B_{2\rho}\setminus B_\rho} |\nabla U|^{p-2}|\nabla\psi|^2\,dx.
  \]
  Hence
  \[
    \begin{aligned}
      &\int_{\mathbb R^n} |\nabla U|^{p-2} |\nabla(\psi_\rho-\psi)|^2\,dx \\
      &\qquad\le C \int_{B_{2\rho}} |\nabla U|^{p-2}|\nabla\psi|^2\,dx + \frac{C}{\rho^2} \int_{B_{2\rho}\setminus B_\rho} |\nabla U|^{p-2}|\psi-c_\rho|^2\,dx \\
      &\qquad\le C \int_{B_{2\rho}} |\nabla U|^{p-2}|\nabla\psi|^2\,dx \to0
    \end{aligned}
  \]
  as \(\rho\downarrow0\).

  By \eqref{eq:psi-weight-norm} and the Cauchy--Schwarz inequality, \(\lim_{r\to\infty}\psi(r\theta)\)
  exists for almost every \(\theta\in\mathbb S^{n-1}\). The limit is zero, since \(\psi\in\dot
  W^{1,p}(\mathbb R^n)\subset L^{p^*}(\mathbb R^n)\). Applying Cauchy--Schwarz to the radial integral
  from \(r\) to infinity gives \(r^{n+\frac{(n-1)(2-p)}{p-1}-2}|\psi(r\theta)|^2\to0\) as
  \(r\to\infty\). The boundary term at infinity therefore vanishes. For \(R\ge1\), integration by parts
  in \(r\) and Cauchy--Schwarz give
  \[
    \begin{aligned}
      \int_{\mathbb R^n\setminus B_R}& |x|^{\frac{(n-1)(2-p)}{p-1}-2}|\psi|^2\,dx \\
      &= \int_{\mathbb S^{n-1}} \int_R^\infty r^{n+\frac{(n-1)(2-p)}{p-1}-3} |\psi(r\theta)|^2\,dr\,d\sigma(\theta) \\
      &\le \frac{4}{ \left( n+\frac{(n-1)(2-p)}{p-1}-2 \right)^2} \int_{\mathbb S^{n-1}} \int_R^\infty r^{n+\frac{(n-1)(2-p)}{p-1}-1} |\partial_r\psi(r\theta)|^2\,dr\,d\sigma(\theta) \\
      &\le C \int_{\mathbb R^n\setminus B_R} |\nabla U|^{p-2}|\nabla\psi|^2\,dx.
    \end{aligned}
  \]

  Choose \(\chi_R\in C_c^\infty(\mathbb R^n)\) such that \(\chi_R=1\) on \(B_R\), \(\chi_R=0\) on
  \(\mathbb R^n\setminus B_{2R}\), and \(|\nabla\chi_R|\le C/R\). Set \(\psi_{\rho,R}=\chi_R\psi_\rho\).
  Since \(0<\rho<1/2\) and \(R\ge1\), \(\psi_\rho=\psi\) on \(\mathbb R^n\setminus B_R\). Thus
  \[
    \begin{aligned}
      \int_{\mathbb R^n}& |\nabla U|^{p-2} |\nabla(\psi_{\rho,R}-\psi_\rho)|^2\,dx \\
      &\le C \int_{\mathbb R^n\setminus B_R} |\nabla U|^{p-2}|\nabla\psi|^2\,dx + \frac{C}{R^2} \int_{B_{2R}\setminus B_R} |\nabla U|^{p-2}|\psi|^2\,dx \\
      &\le C \int_{\mathbb R^n\setminus B_R} |\nabla U|^{p-2}|\nabla\psi|^2\,dx \to0
    \end{aligned}
  \]
  as \(R\to\infty\). Letting first \(R\to\infty\) and then \(\rho\downarrow0\), we obtain
  \(\int_{\mathbb R^n}|\nabla U|^{p-2} |\nabla(\psi_{\rho,R}-\psi)|^2\,dx\to0\).

  The function \(\psi_{\rho,R}\) is compactly supported and constant near the origin. Its gradient is
  supported in a compact annulus on which \(|\nabla U|^{p-2}\) is bounded above and below by positive
  constants. Mollification therefore approximates \(\psi_{\rho,R}\) in the weighted gradient norm by
  functions in \(C_{c,0}^{\infty}(\mathbb R^n)\subset C_{c,0}^{1}(\mathbb R^n)\). Thus \(\psi\in\mathcal
  H_U\) by \eqref{eq:def-HU-low}.

  For \(x\in G_{j,R}^{\eta,\sigma}\) and \(\xi\in\Gamma_\sigma(x)\),
  \begin{equation}
    \label{eq:small-directional-perturbation}
    |\delta_j\partial_\xi\psi_j(x)| \le \delta_j|\nabla\psi_j(x)| \le \eta\sigma|\nabla U(x)| \le \eta|\partial_\xi U(x)|.
  \end{equation}
  If \(|\partial_\xi U| < |\partial_\xi U+\delta_j\partial_\xi\psi_j|\), then \eqref{eq:def-w} and
  \eqref{eq:small-directional-perturbation} imply
  \[
    \begin{aligned}
      0 \le \frac{|w_{\xi,\delta_j}|^{p-2}} {|\partial_\xi U|^{p-2}} -1 &= \frac{ (p-1) \bigl( |\partial_\xi U+\delta_j\partial_\xi\psi_j| -
        |\partial_\xi U| \bigr) }{ (2-p)|\partial_\xi U+\delta_j\partial_\xi\psi_j| + (p-1)|\partial_\xi U| } \\
      &\le (p-1)\eta.
    \end{aligned}
  \]
  If \(|\partial_\xi U+\delta_j\partial_\xi\psi_j| \le|\partial_\xi U|\), then
  \(w_{\xi,\delta_j}=\partial_\xi U\). Hence
  \begin{equation}
    \label{eq:w-ratio-both-cases-low-p}
    |w_{\xi,\delta_j}|^{p-2} \le \bigl(1+(p-1)\eta\bigr) |\partial_\xi U|^{p-2}.
  \end{equation}

  Since
  \[
    \left| \frac{ |\partial_\xi U+\delta_j\partial_\xi\psi_j| - |\partial_\xi U| }{ \delta_j } \right| \le |\partial_\xi\psi_j|,
  \]
  \(p<2\) and \eqref{eq:w-ratio-both-cases-low-p} imply
  \begin{equation}
    \label{eq:B-density-lower-good-region}
    \begin{aligned}
      &\frac p2 |\partial_\xi U|^{p-2} |\partial_\xi\psi_j|^2 + \frac{p(p-2)}2 |w_{\xi,\delta_j}|^{p-2} \left( \frac{ |\partial_\xi U+\delta_j\partial_\xi\psi_j| - |\partial_\xi U| }{ \delta_j } \right)^2 \\
      &\qquad\ge \frac p2 \left( |\partial_\xi U|^{p-2} - (2-p)|w_{\xi,\delta_j}|^{p-2} \right) |\partial_\xi\psi_j|^2 \\
      &\qquad\ge \frac p2 (p-1) \bigl(1-(2-p)\eta\bigr) |\partial_\xi U|^{p-2} |\partial_\xi\psi_j|^2.
    \end{aligned}
  \end{equation}

  The combined integrand in \eqref{eq:def-B} is nonnegative by \eqref{eq:Bge0}. Integrating
  \eqref{eq:B-density-lower-good-region} over \(x\in G_{j,R}^{\eta,\sigma}\) and
  \(\xi\in\Gamma_\sigma(x)\), we obtain
  \begin{equation}
    \label{eq:B-restricted-lower-low-p}
    \begin{aligned}
      &\alpha_{n,p} \left\langle \mathfrak B_{\xi,\delta_j}(\psi_j) \right\rangle_{\mathbb S^{n-1}} \\
      &\qquad\ge \frac p2 (p-1) \bigl(1-(2-p)\eta\bigr) \int_{A_R} \alpha_{n,p} \left\langle \mathbf 1_{\Gamma_\sigma(x)} |\partial_\xi
        U|^{p-2} \left| \xi\cdot \mathbf 1_{G_{j,R}^{\eta,\sigma}} \nabla\psi_j \right|^2 \right\rangle_{\mathbb S^{n-1}} \,dx.
    \end{aligned}
  \end{equation}

  For fixed \(R\) and \(\sigma\), the quadratic form in \eqref{eq:B-restricted-lower-low-p} is
  nonnegative and has bounded measurable coefficients on \(A_R\). By weak lower semicontinuity and
  \eqref{eq:weak-convergence-L2},
  \[
    \begin{aligned}
      &\liminf_{j\to\infty} \alpha_{n,p} \left\langle \mathfrak B_{\xi,\delta_j}(\psi_j) \right\rangle_{\mathbb S^{n-1}} \\
      &\qquad\ge \frac p2 (p-1) \bigl(1-(2-p)\eta\bigr) \int_{A_R} \alpha_{n,p} \left\langle \mathbf 1_{\Gamma_\sigma(x)} |\partial_\xi U|^{p-2} |\partial_\xi\psi|^2 \right\rangle_{\mathbb S^{n-1}} \,dx.
    \end{aligned}
  \]

  Letting \(\eta\downarrow0\) and then \(\sigma\downarrow0\), we obtain by monotone convergence,
  \eqref{eq:angular-identities-low-p}, and \(\alpha_{n,p}=m_{n,p}^{-1}\)
  \begin{equation}
    \label{eq:B-liminf-annulus}
    \begin{aligned}
      &\liminf_{j\to\infty} \alpha_{n,p} \left\langle \mathfrak B_{\xi,\delta_j}(\psi_j) \right\rangle_{\mathbb S^{n-1}} \\
      &\qquad\ge \frac p2(p-1) \int_{A_R} \alpha_{n,p} \left\langle |\partial_\xi U|^{p-2} |\partial_\xi\psi|^2 \right\rangle_{\mathbb S^{n-1}} \,dx \\
      &\qquad= \frac p2 \int_{A_R} |\nabla U|^{p-2} \left( |\nabla\psi|^2 + (p-2) \frac{ (\nabla U\cdot\nabla\psi)^2 }{ |\nabla U|^2 } \right)\,dx.
    \end{aligned}
  \end{equation}
  Now let \(R\to\infty\). Monotone convergence gives \eqref{eq:B-lower-cts}.
\end{proof}

\subsection{\texorpdfstring{The range \(p^*\le2\)}{The range p*<=2}}
\begin{proof}[Proof of Proposition~\ref{prop:coercivity-low-p}]
  By \eqref{eq:principal-equivalence-gap-proof} and \eqref{eq:Lxi-HU-bound}, there is a constant
  \(C_*=C_*(n,p)>0\) such that
  \begin{equation}
    \label{eq:quadratic-terms-controlled-low}
    \begin{aligned}
      \int_{\mathbb R^n} |\nabla U|^{p-2} \left( |\nabla\phi|^2 + (p-2) \frac{(\nabla U\cdot\nabla\phi)^2}{|\nabla U|^2} \right)\,dx +
        \alpha_{n,p} \frac{\tau+1}{pA_0} \operatorname{Var}_{\xi} \bigl(L_\xi(\phi)\bigr) \le C_*\|\phi\|_{\mathcal H_U}^2.
    \end{aligned}
  \end{equation}

  Choose
  \begin{equation}
    \label{eq:coercivity-parameters-low}
    \kappa_0=\kappa_0(n,p)\in(0,1), \qquad C_*\kappa_0\le\frac{c_{\rm sg}}4, \qquad \eta_0=\frac{c_{\rm sg}}{4c_0},
  \end{equation}
  and fix \(0<\kappa\le\kappa_0\). Let \(c_\kappa>0\) be the constant in
  Lemma~\ref{lem:affine-nonlinear-lower-p<2}.

  If \eqref{eq:coercivity-low-p} fails, there are \(\varepsilon_j\downarrow0\) and \(\varphi_j\in\dot
  W^{1,p}(\mathbb R^n)\) with \(\|\nabla\varphi_j\|_{L^p(\mathbb R^n)}=1\),
  \[
    \int_{\mathbb R^n} U^{p^*-2}Z\varphi_j\,dx = 0 \qquad \text{for every }Z\in T_U\mathcal M_{\rm aff},
  \]
  and
  \begin{equation}
    \label{eq:failure-low}
    \begin{aligned}
      &\mathcal N_{{\rm aff}}^{\varepsilon_j,\kappa}(\varphi_j) + \frac{c_\kappa}{2} \alpha_{n,p}\varepsilon_j^{-2} \left\langle \mathfrak
        R_{\xi,\varepsilon_j}(\varphi_j) \right\rangle_{\mathbb S^{n-1}} < \frac p2 \bigl( (p^*-1)\Lambda+\eta_0 \bigr) \mathcal
        P_{\varepsilon_j}(\varphi_j).
    \end{aligned}
  \end{equation}

  Define
  \begin{equation}
    \label{eq:coercivity-normalization-low}
    \rho_j^2 = \int_{\mathbb R^n} \bigl( |\nabla U| + \varepsilon_j|\nabla\varphi_j| \bigr)^{p-2} |\varepsilon_j\nabla\varphi_j|^2\,dx, \qquad \psi_j = \frac{\varepsilon_j}{\rho_j}\varphi_j.
  \end{equation}
  Since \(p<2\),
  \begin{equation}
    \label{eq:rho-to-zero-low}
    \begin{aligned}
      \rho_j^2 &\le \int_{\mathbb R^n} |\varepsilon_j\nabla\varphi_j|^p\,dx = \varepsilon_j^p \|\nabla\varphi_j\|_{L^p(\mathbb R^n)}^p = \varepsilon_j^p \to0.
    \end{aligned}
  \end{equation}
  Moreover,
  \begin{equation}
    \label{eq:two-scale-normalization-low}
    \int_{\mathbb R^n} \bigl( |\nabla U| + \rho_j|\nabla\psi_j| \bigr)^{p-2} |\nabla\psi_j|^2\,dx = 1.
  \end{equation}
  Since \(\psi_j=(\varepsilon_j/\rho_j)\varphi_j\),
  \begin{equation}
    \label{eq:orth-normalization}
    \int_{\mathbb R^n} U^{p^*-2}Z\psi_j\,dx = 0 \qquad \text{for every }Z\in T_U\mathcal M_{\rm aff}.
  \end{equation}

  By \eqref{eq:failure-low} and \eqref{eq:coercivity-normalization-low},
  \begin{equation}
    \label{eq:failure-low-normalized}
    \begin{aligned}
      &\mathcal N_{{\rm aff}}^{\rho_j,\kappa}(\psi_j) + \frac{c_\kappa}{2} \alpha_{n,p}\rho_j^{-2} \left\langle \mathfrak
        R_{\xi,\rho_j}(\psi_j) \right\rangle_{\mathbb S^{n-1}} < \frac p2 \bigl( (p^*-1)\Lambda+\eta_0 \bigr) \mathcal P_{\rho_j}(\psi_j).
    \end{aligned}
  \end{equation}

  By \eqref{eq:rho-to-zero-low}, \eqref{eq:two-scale-normalization-low}, and \cite[Lemma~3.4]{FZ22},
  after passing to a subsequence, there is \(\psi\in\dot W^{1,p}(\mathbb R^n) \cap L^2(\mathbb
  R^n;U^{p^*-2}dx)\) such that
  \begin{equation}
    \label{eq:weak-convergence-low}
    \psi_j \rightharpoonup \psi \qquad \text{weakly in }\dot W^{1,p}(\mathbb R^n),
  \end{equation}
  and
  \begin{equation}
    \label{eq:potential-convergence-low}
    \mathcal P_{\rho_j}(\psi_j) \to \int_{\mathbb R^n} U^{p^*-2}\psi^2\,dx.
  \end{equation}

  For \(Z\in T_U\mathcal M_{\rm aff}\), the bound \(|Z|\le C_ZU\) implies \(U^{p^*-2}Z\in
  L^{(p^*)'}(\mathbb R^n)\). By Sobolev embedding, \eqref{eq:weak-convergence-low}, and
  \eqref{eq:orth-normalization},
  \begin{equation}
    \label{eq:limit-orthogonality-low}
    \psi \perp T_U\mathcal M_{\rm aff}.
  \end{equation}

  The map \(\xi\mapsto L_\xi\) is continuous from \(\mathbb S^{n-1}\) into \((\dot W^{1,p})^*\). Its
  image is compact in the dual norm. Thus \eqref{eq:weak-convergence-low} gives \(\sup_{\xi\in\mathbb
  S^{n-1}}|L_\xi(\psi_j)-L_\xi(\psi)|\to0\), and hence
  \begin{equation}
    \label{eq:variance-convergence-low}
    \operatorname{Var}_{\xi} \bigl(L_\xi(\psi_j)\bigr) \to \operatorname{Var}_{\xi} \bigl(L_\xi(\psi)\bigr).
  \end{equation}

  By Lemma~\ref{lem:averaged-affine-lsc-low-p}, \(\psi\in\mathcal H_U\). Using
  \eqref{eq:quadratic-terms-controlled-low} and \eqref{eq:coercivity-parameters-low}, we deduce from
  \eqref{eq:B-lower-cts} and \eqref{eq:variance-convergence-low} that
  \begin{equation}
    \label{eq:limit-estimate-low}
    \begin{aligned}
      \liminf_{j\to\infty}& \frac2p \mathcal N_{{\rm aff}}^{\rho_j,\kappa}(\psi_j) \\
      &\ge (1-\kappa) \int_{\mathbb R^n} |\nabla U|^{p-2} \left( |\nabla\psi|^2 + (p-2) \frac{(\nabla U\cdot\nabla\psi)^2}{|\nabla U|^2} \right)\,dx \\
      &\qquad - (1+\kappa) \alpha_{n,p}\frac{\tau+1}{pA_0} \operatorname{Var}_{\xi}\bigl(L_\xi(\psi)\bigr) \\
      &\ge \int_{\mathbb R^n} |\nabla U|^{p-2} \left( |\nabla\psi|^2 + (p-2) \frac{(\nabla U\cdot\nabla\psi)^2}{|\nabla U|^2} \right)\,dx \\
      &\qquad - \alpha_{n,p}\frac{\tau+1}{pA_0} \operatorname{Var}_{\xi}\bigl(L_\xi(\psi)\bigr) - \frac{c_{\rm sg}}4\|\psi\|_{\mathcal H_U}^2 \\
      &= Q_{\rm aff,p}(\psi) + (p^*-1)\Lambda \int_{\mathbb R^n}U^{p^*-2}\psi^2\,dx - \frac{c_{\rm sg}}4\|\psi\|_{\mathcal H_U}^2 \\
      &\ge \left( (p^*-1)\Lambda+\frac{3c_{\rm sg}}{4c_0} \right) \int_{\mathbb R^n}U^{p^*-2}\psi^2\,dx.
    \end{aligned}
  \end{equation}
  Here we used Proposition~\ref{prop:affine-spectral-gap} and \eqref{eq:weighted-Poincare-used-low}.

  By \eqref{eq:failure-low-normalized} and \eqref{eq:potential-convergence-low}, we obtain
  \begin{equation}
    \label{eq:upper-limit-N}
    \limsup_{j\to\infty} \frac2p \mathcal N_{{\rm aff}}^{\rho_j,\kappa}(\psi_j) \le \bigl( (p^*-1)\Lambda+\eta_0 \bigr) \int_{\mathbb R^n} U^{p^*-2}\psi^2\,dx.
  \end{equation}
  Then \eqref{eq:coercivity-parameters-low}, \eqref{eq:limit-estimate-low}, and \eqref{eq:upper-limit-N}
  imply \(\psi=0\).

  By \eqref{eq:potential-convergence-low} and \eqref{eq:variance-convergence-low} give \(\mathcal
  P_{\rho_j}(\psi_j)\to0\) and \(\operatorname{Var}_{\xi}(L_\xi(\psi_j))\to0\). By \eqref{eq:def-N} and
  \eqref{eq:failure-low-normalized},
  \begin{equation}
    \label{eq:B-R-vanish-low}
    \begin{aligned}
      &\alpha_{n,p}(1-\kappa) \left\langle \mathfrak B_{\xi,\rho_j}(\psi_j) \right\rangle_{\mathbb S^{n-1}} + \frac{c_\kappa}{2}
        \alpha_{n,p}\rho_j^{-2} \left\langle \mathfrak R_{\xi,\rho_j}(\psi_j) \right\rangle_{\mathbb S^{n-1}} \\
      &\qquad< \frac p2 \bigl( (p^*-1)\Lambda+\eta_0 \bigr) \mathcal P_{\rho_j}(\psi_j) + \alpha_{n,p}(1+\kappa) \frac{\tau+1}{2A_0} \operatorname{Var}_{\xi} \bigl(L_\xi(\psi_j)\bigr) \to0.
    \end{aligned}
  \end{equation}
  It follows from \eqref{eq:B-R-vanish-low} that
  \begin{equation}
    \label{eq:B-R-combined-vanish-low}
    \left\langle \mathfrak B_{\xi,\rho_j}(\psi_j) \right\rangle_{\mathbb S^{n-1}} + \rho_j^{-2} \left\langle \mathfrak R_{\xi,\rho_j}(\psi_j) \right\rangle_{\mathbb S^{n-1}} \to0.
  \end{equation}
  By \eqref{eq:averaged-R-lower}, \eqref{eq:two-scale-normalization-low}, and
  \eqref{eq:B-R-combined-vanish-low},
  \[
    \begin{aligned}
      1 &= \int_{\mathbb R^n} \bigl( |\nabla U| + \rho_j|\nabla\psi_j| \bigr)^{p-2} |\nabla\psi_j|^2\,dx \\
      &\le C \left[ \left\langle \mathfrak B_{\xi,\rho_j}(\psi_j) \right\rangle_{\mathbb S^{n-1}} + \rho_j^{-2} \left\langle \mathfrak R_{\xi,\rho_j}(\psi_j) \right\rangle_{\mathbb S^{n-1}} \right] \rightarrow0.
    \end{aligned}
  \]
  This contradiction proves the proposition.
\end{proof}

\subsection{\texorpdfstring{The range \(p^*>2\)}{The range p*>2}}

\begin{proof}[Proof of Proposition~\ref{prop:coercivity-high-p}]
  Choose \(\kappa_0\) and \(\eta_0\) by \eqref{eq:coercivity-parameters-low}. Fix
  \(0<\kappa\le\kappa_0\), and let \(c_\kappa\) be the constant in
  Lemma~\ref{lem:affine-nonlinear-lower-p<2}.

  If \eqref{eq:coercivity-high-p} fails, there are \(\varepsilon_j\downarrow0\) and \(\varphi_j\in\dot
  W^{1,p}(\mathbb R^n)\) with \(\|\nabla\varphi_j\|_{L^p(\mathbb R^n)}=1\), \(\varphi_j\perp T_U\mathcal
  M_{\rm aff}\), and
  \begin{equation}
    \label{eq:failure-high-p}
    \begin{aligned}
      \mathcal N_{{\rm aff}}^{\varepsilon_j,\kappa}(\varphi_j) + \frac{c_\kappa}{2}\alpha_{n,p}\varepsilon_j^{-2} \left\langle \mathfrak
        R_{\xi,\varepsilon_j}(\varphi_j) \right\rangle_{\mathbb S^{n-1}} < \frac p2 \bigl((p^*-1)\Lambda+\eta_0\bigr) \int_{\mathbb R^n}
        U^{p^*-2}\varphi_j^2\,dx.
    \end{aligned}
  \end{equation}
  Define \(\rho_j\) and \(\psi_j\) by \eqref{eq:coercivity-normalization-low}. Then
  \(\rho_j^2\le\varepsilon_j^p\to0\) by \eqref{eq:rho-to-zero-low}, and
  \begin{equation}
    \label{eq:high-p-normalization}
    \int_{\mathbb R^n} \bigl( |\nabla U|+\rho_j|\nabla\psi_j| \bigr)^{p-2} |\nabla\psi_j|^2\,dx = 1,
  \end{equation}
  and \(\psi_j\perp T_U\mathcal M_{\rm aff}\).

  By H\"older's inequality, \eqref{eq:high-p-normalization}, and Young's inequality yield
  \begin{align*}
    \int_{\mathbb R^n}|\nabla\psi_j|^p\,dx &\le \left( \int_{\mathbb R^n} \bigl( |\nabla U|+\rho_j|\nabla\psi_j| \bigr)^{p-2} |\nabla\psi_j|^2\,dx \right)^{p/2} \\
    &\qquad\cdot \left( \int_{\mathbb R^n} \bigl( |\nabla U|+\rho_j|\nabla\psi_j| \bigr)^p\,dx \right)^{(2-p)/2} \\
    &\le C \left( 1+ \rho_j^p \int_{\mathbb R^n} |\nabla\psi_j|^p\,dx \right)^{(2-p)/2} \\
    &\le C + C\rho_j^{\frac{p(2-p)}2} \left( \int_{\mathbb R^n} |\nabla\psi_j|^p\,dx \right)^{\frac{2-p}2} \\
    & \le C+ \frac12 \int_{\mathbb R^n} |\nabla\psi_j|^p\,dx + C\rho_j^{2-p}.
  \end{align*}
  It follows that \(\sup_j\|\nabla\psi_j\|_{L^p(\mathbb R^n)}<\infty\). After passing to a subsequence,
  \begin{equation}
    \label{eq:weak-convergence-high}
    \psi_j\rightharpoonup\psi \qquad \text{weakly in }\dot W^{1,p}(\mathbb R^n)
  \end{equation}
  for some \(\psi\in\dot W^{1,p}(\mathbb R^n)\). Lemma~\ref{lem:compact-embedding-high-p} gives
  \(\psi_j\to\psi\) strongly in \(L^2(\mathbb R^n;U^{p^*-2}dx)\). Hence
  \begin{equation}
    \label{eq:high-p-potential-limit}
    \int_{\mathbb R^n} U^{p^*-2}\psi_j^2\,dx \to \int_{\mathbb R^n} U^{p^*-2}\psi^2\,dx, \qquad \psi\perp T_U\mathcal M_{\rm aff}.
  \end{equation}

  By \eqref{eq:coercivity-normalization-low} and \eqref{eq:failure-high-p},
  \begin{equation}
    \label{eq:failure-high-p-normalized}
    \begin{aligned}
      \mathcal N_{{\rm aff}}^{\rho_j,\kappa}(\psi_j) + \frac{c_\kappa}{2}\alpha_{n,p}\rho_j^{-2} \left\langle \mathfrak
        R_{\xi,\rho_j}(\psi_j) \right\rangle_{\mathbb S^{n-1}} < \frac p2 \bigl((p^*-1)\Lambda+\eta_0\bigr) \int_{\mathbb R^n}
        U^{p^*-2}\psi_j^2\,dx.
    \end{aligned}
  \end{equation}

  By \eqref{eq:weak-convergence-high} and the compactness of \(\{L_\xi:\xi\in\mathbb S^{n-1}\}\) in
  \((\dot W^{1,p})^*\),
  \begin{equation}
    \label{eq:variance-convergence-high}
    \operatorname{Var}_{\xi} \bigl(L_\xi(\psi_j)\bigr) \to \operatorname{Var}_{\xi} \bigl(L_\xi(\psi)\bigr).
  \end{equation}
  By Lemma~\ref{lem:averaged-affine-lsc-low-p}, \(\psi\in\mathcal H_U\). Using
  \eqref{eq:variance-convergence-high} and \eqref{eq:high-p-potential-limit} in
  \eqref{eq:limit-estimate-low}, we obtain
  \begin{equation}
    \label{eq:limit-estimate-higher}
    \liminf_{j\to\infty} \frac2p \mathcal N_{{\rm aff}}^{\rho_j,\kappa}(\psi_j) \ge \left( (p^*-1)\Lambda + \frac{3c_{\rm sg}}{4c_0} \right) \int_{\mathbb R^n}U^{p^*-2}\psi^2\,dx.
  \end{equation}

  By \eqref{eq:failure-high-p-normalized} and \eqref{eq:high-p-potential-limit}, we have
  \begin{equation}
    \label{eq:upper-estimate-high}
    \limsup_{j\to\infty} \frac2p \mathcal N_{{\rm aff}}^{\rho_j,\kappa}(\psi_j) \le \bigl((p^*-1)\Lambda+\eta_0\bigr) \int_{\mathbb R^n} U^{p^*-2}\psi^2\,dx.
  \end{equation}
  Equations~\eqref{eq:coercivity-parameters-low}, \eqref{eq:limit-estimate-higher}, and
  \eqref{eq:upper-estimate-high} imply \(\psi=0\).

  Equations~\eqref{eq:high-p-potential-limit} and \eqref{eq:variance-convergence-high} give
  \[
    \int_{\mathbb R^n} U^{p^*-2}\psi_j^2\,dx \to0, \qquad \operatorname{Var}_{\xi} \bigl(L_\xi(\psi_j)\bigr) \to0.
  \]
  By \eqref{eq:def-N} and \eqref{eq:failure-high-p-normalized},
  \begin{equation}
    \label{eq:BR-vanish-high}
    \left\langle \mathfrak B_{\xi,\rho_j}(\psi_j) \right\rangle_{\mathbb S^{n-1}} + \rho_j^{-2} \left\langle \mathfrak R_{\xi,\rho_j}(\psi_j) \right\rangle_{\mathbb S^{n-1}} \to0.
  \end{equation}

  Equations~\eqref{eq:averaged-R-lower} and \eqref{eq:BR-vanish-high} contradict
  \eqref{eq:high-p-normalization}.
\end{proof}

\section{Qualitative compactness and proof of Theorem~\ref{thm:intro-main-affine-p<2}}
\label{sec:proof-main-theorem}

\subsection{Qualitative compactness}

\begin{proposition}
  \label{prop:qualitative-affine-compactness-low-p}
  Let \(1<p<n\), and let \(u_k\in\dot W^{1,p}(\mathbb R^n)\) satisfy \(\|u_k\|_{L^{p^*}(\mathbb
  R^n)}=1\) and \(\delta_{\rm aff}(u_k)\to0\). Then there exist \(c_k\ne0\), \(A_k\in SL(n)\),
  \(\lambda_k>0\), and \(x_k\in\mathbb R^n\) such that
  \[
    c_k^{-1}T_{\lambda_kA_k,x_k}u_k \to U \qquad \text{strongly in }\dot W^{1,p}(\mathbb R^n).
  \]
\end{proposition}

\begin{proof}
  For \(2\le p<n\), the result follows from \cite[Theorem~3.1]{FLY26}. Assume \(1<p<2\). We use the
  affine normalization in \cite[Lemma~4.1]{FLZ26B} and apply the Sobolev profile decomposition
  \cite[Lemma~3.2]{FLY26}. The directional estimate and the \(L^{p^*}(\mathbb R^n)\)-norm identity in
  \cite[equations~(3.2) and~(3.9)]{FLY26}, together with the sharp affine inequality, show that exactly
  one profile is nonzero and that it is an extremal. This part of the proof of \cite[Theorem~3.1]{FLY26}
  is valid for every \(1<p<n\).

  Thus, after passing to a subsequence, we can choose \(c_k\ne0\), \(A_k\in SL(n)\), \(\lambda_k>0\),
  and \(x_k\in\mathbb R^n\) such that
  \[
    c_k^{-1}T_{\lambda_kA_k,x_k}u_k =U+\widetilde\rho_k,
  \]
  where
  \[
    \widetilde\rho_k\rightharpoonup0 \quad\text{in }\dot W^{1,p}(\mathbb R^n), \qquad \|\widetilde\rho_k\|_{L^{p^*}(\mathbb R^n)}\to0,
      \qquad \mathcal E_{\rm aff,p}(U+\widetilde\rho_k)^p \to\mathcal E_{\rm aff,p}(U)^p.
  \]
  It remains to show that \(\|\nabla\widetilde\rho_k\|_{L^p(\mathbb R^n)}\to0\). For \(\xi\in\mathbb
  S^{n-1}\), set
  \[
    R_k(\xi) = \int_{\mathbb R^n} \min \left\{ |\partial_\xi\widetilde\rho_k|^p,\, |\partial_\xi U|^{p-2} |\partial_\xi\widetilde\rho_k|^2 \right\}\,dx.
  \]
  By \cite[Lemma~2.1(i)]{FZ22},
  \[
    |a+b|^p \ge |a|^p + p|a|^{p-2}ab + c_p \min \left\{ |b|^p,\, |a|^{p-2}|b|^2 \right\}, \qquad 1<p<2,
  \]
  we obtain
  \[
    A_\xi(U+\widetilde\rho_k) \ge A_0 + L_\xi(\widetilde\rho_k) + c_pR_k(\xi).
  \]
  As in the proof of \cite[Proposition~4.1]{FLZ26B}, the compactness of \(\{L_\xi:\xi\in\mathbb
  S^{n-1}\}\) in \((\dot W^{1,p}(\mathbb R^n))^*\) and the weak convergence
  \(\widetilde\rho_k\rightharpoonup0\) imply \(\eta_k = \sup_{\xi\in\mathbb S^{n-1}}
  |L_\xi(\widetilde\rho_k)| \to0\). Hence
  \begin{equation}
    \label{eq:directional-two-scale-global-lower-low-p}
    A_\xi(U+\widetilde\rho_k) \ge A_0-\eta_k+c_pR_k(\xi).
  \end{equation}

  We claim that
  \begin{equation}
    \label{eq:averaged-two-scale-remainder-global}
    \left\langle R_k \right\rangle_{\mathbb S^{n-1}} \to0.
  \end{equation}
  Indeed, since \((\widetilde\rho_k)\) is bounded in \(\dot W^{1,p}(\mathbb R^n)\), \(0\le R_k(\xi) \le
  A_\xi(\widetilde\rho_k) \le C\) uniformly in \(k\) and \(\xi\). By
  \eqref{eq:directional-two-scale-global-lower-low-p},
  \[
    \begin{aligned}
      \mathcal E_{\rm aff,p}(U+\widetilde\rho_k)^p &= \alpha_{n,p} \Phi\bigl(A_\xi(U+\widetilde\rho_k)\bigr) \ge \alpha_{n,p} \Phi\bigl(A_0-\eta_k+c_pR_k(\xi)\bigr).
    \end{aligned}
  \]
  The mean value theorem therefore gives
  \[
    \left\langle (A_0-\eta_k+c_pR_k)^{-n/p} \right\rangle_{\mathbb S^{n-1}} \le (A_0-\eta_k)^{-n/p} -c\left\langle R_k\right\rangle_{\mathbb S^{n-1}}.
  \]
  If \eqref{eq:averaged-two-scale-remainder-global} failed, then, after passing to a subsequence, there
  would exist \(\delta>0\) such that \(\langle R_k\rangle_{\mathbb S^{n-1}}\ge\delta\). Consequently,
  \[
    \begin{aligned}
      \liminf_{k\to\infty} \mathcal E_{\rm aff,p}(U+\widetilde\rho_k)^p &\ge \alpha_{n,p} \left( A_0^{-n/p}-c\delta \right)^{-p/n} > \alpha_{n,p}A_0 = \mathcal E_{\rm aff,p}(U)^p,
    \end{aligned}
  \]
  contrary to \(\mathcal E_{\rm aff,p}(U+\widetilde\rho_k)^p \to\mathcal E_{\rm aff,p}(U)^p\). This
  proves \eqref{eq:averaged-two-scale-remainder-global}.

  Since \((\widetilde\rho_k)\) is bounded in \(\dot W^{1,p}(\mathbb R^n)\), by
  \eqref{eq:averaged-R-lower}, \eqref{eq:averaged-two-scale-remainder-global} and H\"older's inequality
  \[
    \begin{aligned}
      \int_{\mathbb R^n} |\nabla\widetilde\rho_k|^p\,dx &\le \left( \int_{\mathbb R^n} \bigl( |\nabla U|+|\nabla\widetilde\rho_k|
        \bigr)^{p-2} |\nabla\widetilde\rho_k|^2\,dx \right)^{p/2} \left( \int_{\mathbb R^n} \bigl( |\nabla U|+|\nabla\widetilde\rho_k|
        \bigr)^p\,dx \right)^{(2-p)/2} \\
      &\le C \left\langle R_k \right\rangle_{\mathbb S^{n-1}}^{p/2} \rightarrow0.
    \end{aligned}
  \]
  This completes the proof.
\end{proof}

\subsection{Choice of an affine extremal}

Orthogonality is understood as in \eqref{eq:orthogonality-convention}.
\begin{lemma}
  \label{lem:affine-modulation-low-p}
  Let \(1<p<2\). There exist \(\eta_0>0\), \(R_0>0\), and a modulus \(\omega:[0,\eta_0]\to[0,\infty)\),
  with \(\omega(t)\to0\) as \(t\downarrow0\), such that the following holds. If \(w\in\dot
  W^{1,p}(\mathbb R^n)\) and \(\|\nabla(w-U)\|_{L^p(\mathbb R^n)}\le\eta_0\), then there exists
  \(V\in\mathcal M_{\rm aff}\) satisfying
  \[
    \|\nabla(V-U)\|_{L^p(\mathbb R^n)}\le R_0, \qquad \|\nabla(w-V)\|_{L^p(\mathbb R^n)} \le \omega\bigl(\|\nabla(w-U)\|_{L^p(\mathbb R^n)}\bigr),
  \]
  and
  \begin{equation}
    \label{eq:modulation-orthogonality-low-p}
    \int_{\mathbb R^n}|V|^{p^*-2}Z(w-V)\,dx=0 \qquad \text{for every }Z\in T_V\mathcal M_{\rm aff}.
  \end{equation}
\end{lemma}

\begin{proof}
  Near the nonzero extremal \(U\), use the local parameterization of the affine extremal cone from the
  proof of \cite[Lemma~4.2]{FLZ26B}; rotational parameters are omitted because \(U\) is radial. For
  \(w\) close to \(U\), minimize over a sufficiently small closed parameter ball the functional
  \[
    V\mapsto \frac1{p^*}\int_{\mathbb R^n}|V|^{p^*}\,dx - \frac1{p^*-1}\int_{\mathbb R^n}|V|^{p^*-2}Vw\,dx.
  \]
  The argument of \cite[Lemma~4.1]{FZ22} shows that, after decreasing \(\eta_0\), the minimizer is an
  interior point, converges to \(U\) as \(w\to U\), and satisfies the stated modulus estimate.
  Differentiating the functional along any smooth curve in \(\mathcal M_{\rm aff}\) through the
  minimizer gives \eqref{eq:modulation-orthogonality-low-p}. The local parameterization and the
  minimization argument above remain valid for \(1<p<2\); no assumption \(p\ge2\) is needed.
\end{proof}

\subsection{Proof of Theorem~\ref{thm:intro-main-affine-p<2}}

\begin{proof}[Proof of Theorem~\ref{thm:intro-main-affine-p<2}]
  Let \(0\ne u\in\dot W^{1,p}(\mathbb R^n)\) and \(d=d_{\rm aff}(u,\mathcal M_{\rm aff})\). We first
  prove
  \begin{equation}
    \label{eq:local-normalized-affine-low-p}
    d\le d_0 \quad\implies\quad \frac{\mathcal E_{\rm aff,p}(u)} {S_{n,p}\|u\|_{L^{p^*}(\mathbb R^n)}}-1 \ge c_0d^2,
  \end{equation}
  for some \(d_0,c_0>0\) depending only on \(n,p\). The case \(d=0\) follows from
  \eqref{eq:affine-Sobolev}. For \(0<d\le d_0\), \eqref{eq:affine-distance} gives \(a\in\mathbb R\),
  \(A\in SL(n)\), \(\lambda>0\), and \(x_0\in\mathbb R^n\) such that, with \(T_0=T_{\lambda A,x_0}\),
  \begin{equation}
    \label{eq:bdd-p}
    \|\nabla(T_0u-aU)\|_{L^p(\mathbb R^n)} \le 2d\|\nabla T_0u\|_{L^p(\mathbb R^n)}.
  \end{equation}
  By \eqref{eq:bdd-p} with \(d_0<1/4\), \(a\ne0\) and set \(w=a^{-1}T_0u\), we have
  \[
    \|\nabla(w-U)\|_{L^p(\mathbb R^n)} \le \frac{2d}{1-2d}\|\nabla U\|_{L^p(\mathbb R^n)}.
  \]
  For small \(d_0\), Lemma~\ref{lem:affine-modulation-low-p} gives \(V\in\mathcal M_{\rm aff}\)
  satisfying \eqref{eq:modulation-orthogonality-low-p}. Write \(V=bT_{\mu B,y}^{-1}U\), where \(b\ne0\)
  and \(\mu B\) is close to the identity. Define \(h\) by
  \begin{equation}
    \label{eq:def-modulated-remainder-low-p}
    b^{-1}T_{\mu B,y}w=U+h.
  \end{equation}
  For \(V\) near \(U\),
  \[
    \|\nabla h\|_{L^p(\mathbb R^n)} \le C\|\nabla(w-V)\|_{L^p(\mathbb R^n)} \le C\,\omega\bigl(\|\nabla(w-U)\|_{L^p(\mathbb R^n)}\bigr).
  \]
  For every \(Z\in T_U\mathcal M_{\rm aff}\), \(bT_{\mu B,y}^{-1}Z\in T_V\mathcal M_{\rm aff}\).
  Changing variables in \eqref{eq:modulation-orthogonality-low-p} gives \(h\perp T_U\mathcal M_{\rm
  aff}\).

  For \(h\ne0\), write \(h=\varepsilon\varphi\), where \(\varepsilon=\|\nabla h\|_{L^p(\mathbb R^n)}\).
  Then \(\|\nabla\varphi\|_{L^p(\mathbb R^n)}=1\) and \(\varphi\perp T_U\mathcal M_{\rm aff}\). For
  small \(d_0\), by \eqref{eq:local-lower-low-p} when \(1<p\le2n/(n+2)\) and
  \eqref{eq:local-lower-high-p} when \(2n/(n+2)<p<2\), \(\delta_{\rm aff}(U+h)\ge c\varepsilon^2
  =c\|\nabla h\|_{L^p(\mathbb R^n)}^2\).

  By \eqref{eq:affine-distance}, affine invariance, and \eqref{eq:def-modulated-remainder-low-p},
  \begin{equation}
    \label{eq:distance-controlled-by-h-low-p}
    d \le \frac{\|\nabla h\|_{L^p(\mathbb R^n)}} {\|\nabla(U+h)\|_{L^p(\mathbb R^n)}} \le C\|\nabla h\|_{L^p(\mathbb R^n)}.
  \end{equation}
  Since \(\|U+h\|_{L^{p^*}(\mathbb R^n)}\) and \(\mathcal E_{\rm aff,p}(U+h)\) are bounded above and
  away from zero for small \(d_0\),
  \begin{equation}
    \label{sec4edit:normalized-deficit-bound}
    \begin{aligned}
      \frac{\mathcal E_{\rm aff,p}(U+h)} {S_{n,p}\|U+h\|_{L^{p^*}(\mathbb R^n)}}-1 &\ge c\left[ \left( \frac{\mathcal E_{\rm aff,p}(U+h)} {S_{n,p}\|U+h\|_{L^{p^*}(\mathbb R^n)}} \right)^p-1 \right] \\
      &= c\, \frac{\delta_{\rm aff}(U+h)} {S_{n,p}^p\|U+h\|_{L^{p^*}(\mathbb R^n)}^p} \ge c\|\nabla h\|_{L^p(\mathbb R^n)}^2.
    \end{aligned}
  \end{equation}
  Then \eqref{eq:distance-controlled-by-h-low-p} and \eqref{sec4edit:normalized-deficit-bound} imply
  \eqref{eq:local-normalized-affine-low-p}.

  If \eqref{eq:intro-main-affine-p<2} fails, there are \(0\ne u_k\in\dot W^{1,p}(\mathbb R^n)\) such
  that
  \begin{equation}
    \label{eq:global-failure-low-p}
    \frac{\mathcal E_{\rm aff,p}(u_k)} {S_{n,p}\|u_k\|_{L^{p^*}(\mathbb R^n)}}-1 < \frac1k d_{\rm aff}(u_k,\mathcal M_{\rm aff})^2.
  \end{equation}
  By homogeneity, assume \(\|u_k\|_{L^{p^*}(\mathbb R^n)}=1\). Since \(d_{\rm aff}(u_k,\mathcal M_{\rm
  aff})\le1\), \eqref{eq:global-failure-low-p} gives \(\delta_{\rm aff}(u_k)\to0\).
  Proposition~\ref{prop:qualitative-affine-compactness-low-p} and \eqref{eq:affine-distance} give
  \(d_{\rm aff}(u_k,\mathcal M_{\rm aff})\to0\). For large \(k\),
  \eqref{eq:local-normalized-affine-low-p} contradicts \eqref{eq:global-failure-low-p}.

  For optimality, choose a nonzero \(\phi\in C_c^\infty(\{1<|x|<2\})\) satisfying \(\phi(-x)=-\phi(x)\),
  and set \(u_\varepsilon=U+\varepsilon\phi\). By \eqref{eq:affine-deficit-expansion-low-p},
  \begin{equation}
    \label{eq:sharpness-normalized-deficit-low-p}
    \frac{\mathcal E_{\rm aff,p}(u_\varepsilon)} {S_{n,p}\|u_\varepsilon\|_{L^{p^*}(\mathbb R^n)}}-1 \le C\varepsilon^2.
  \end{equation}

  The triangle and Sobolev inequalities in \eqref{eq:affine-distance} give
  \[
    d_{\rm aff}(u,\mathcal M_{\rm aff}) \ge \frac{\inf_{V\in\mathcal M_{\rm aff}} \|u-V\|_{L^{p^*}(\mathbb R^n)}}{3\|u\|_{L^{p^*}(\mathbb R^n)}}.
  \]
  We show that \(inf_{V\in\mathcal M_{\rm aff}} \|u_\varepsilon-V\|_{L^{p^*}(\mathbb R^n)}\ge
  c\varepsilon\). Otherwise, there exist \(\varepsilon_j\downarrow0\) and \(V_j\in\mathcal M_{\rm aff}\)
  such that \(\|u_{\varepsilon_j}-V_j\|_{L^{p^*}(\mathbb R^n)}=o(\varepsilon_j)\). Each \(V_j\) is
  symmetric about some point \(y_j\), so
  \[
    \|u_{\varepsilon_j} -u_{\varepsilon_j}(2y_j-\cdot)\|_{L^{p^*}(\mathbb R^n)} =o(\varepsilon_j).
  \]
  Since \(u_{\varepsilon_j}\to U\) and \(U\) is strictly radially decreasing, this implies \(y_j\to0\).
  On \(3<|x|<4\), both perturbation terms vanish for large \(j\). Taylor's formula and \(U'<0\)
  therefore yield
  \[
    |y_j| \le C\|U-U(2y_j-\cdot)\|_{L^{p^*}(\{3<|x|<4\})} =o(\varepsilon_j).
  \]
  Using \(U\in W^{1,p^*}\) and the oddness of \(\phi\), we obtain
  \[
    u_{\varepsilon_j} -u_{\varepsilon_j}(2y_j-\cdot) =2\varepsilon_j\phi+o(\varepsilon_j) \qquad\text{in }L^{p^*},
  \]
  a contradiction. Hence
  \begin{equation}
    \label{eq:sharpness-distance-low-p}
    d_{\rm aff}(u_\varepsilon,\mathcal M_{\rm aff}) \ge c\varepsilon.
  \end{equation}
  If \eqref{eq:intro-main-affine-p<2} held with an exponent \(0<q<2\), these estimates would give
  \(c\varepsilon^q\le C\varepsilon^2\), which is impossible as \(\varepsilon\downarrow0\).
\end{proof}

\appendix

\section{\texorpdfstring{The affine kernel in dimension \(n=2\)}{The affine kernel in dimension n=2}}
\label{app:kernel-affine-Hessian-n=2}

The case \(n\ge3\) was treated in Proposition~\ref{prop:affine-kernel-n-ge3-low-p}. We now consider
\(n=2\). We use Fourier series on \(\mathbb S^1\); see also \cite[Section~5.3]{FLZ26A} for the affine
fractional Sobolev inequality. The radial identities of \cite[Lemmas~5.1--5.3]{FLZ26B}  remain valid for
\(1<p<2\).

\begin{proposition}
  \label{prop:affine-kernel-n=2}
  Let \(n=2\) and \(1<p<2\). Then
  \begin{equation}
    \label{eq:affine-kernel-n=2}
    \ker_{\mathcal H_U}Q_{\rm aff,p} = \operatorname{span} \left\{ U,\, Z_0,\, \partial_{x_1}U,\, \partial_{x_2}U,\, rU'(r)\cos(2\vartheta),\, rU'(r)\sin(2\vartheta) \right\} = T_U\mathcal M_{\rm aff},
  \end{equation}
  where \(Z_0=(2-p)U/p+rU'\).
\end{proposition}

\begin{proof}
  Write \(x=r\theta\), \(\theta=(\cos\vartheta,\sin\vartheta)\). The orthonormal Fourier basis for
  \(d\vartheta/(2\pi)\) is \(Y_{0,1}=1\) and
  \[
    Y_{\ell,1}(\vartheta)=\sqrt2\cos(\ell\vartheta), \qquad Y_{\ell,2}(\vartheta)=\sqrt2\sin(\ell\vartheta), \qquad \ell\ge1.
  \]
  Then \(-\Delta_{\mathbb S^1}Y_{\ell,m}=\ell^2Y_{\ell,m}\). For a finite Fourier expansion, write
  \begin{equation}
    \label{eq:spherical-decomposition-n=2}
    \phi(r,\theta) = f_{0,1}(r) + \sum_{\ell=1}^{\infty} \sum_{m=1}^{2} f_{\ell,m}(r)Y_{\ell,m}(\theta).
  \end{equation}

  Since \(U\) is radial, \(Q_{\rm Sob,p}\) is diagonal in \eqref{eq:spherical-decomposition-n=2}. Its
  radial part is
  \[
    \begin{aligned}
      Q_{\rm Sob,p}^{(0)}(f) &= 2\pi(p-1) \int_0^\infty |U'|^{p-2}|f'|^2r\,dr - 2\pi(p^*-1)\Lambda \int_0^\infty U^{p^*-2}f^2r\,dr \\
      &\quad+ 2\pi(p^*-p)\Lambda \left( \int_0^\infty U^{p^*}r\,dr \right)^{-1} \left( \int_0^\infty U^{p^*-1}fr\,dr \right)^2.
    \end{aligned}
  \]
  For every \(\ell\ge1\),
  \[
    \begin{aligned}
      Q_{\rm Sob,p}^{(\ell)}(f) &= 2\pi \int_0^\infty |U'|^{p-2} \left( (p-1)|f'|^2 + \ell^2\frac{f^2}{r^2} \right)r\,dr \\
      &\quad- 2\pi(p^*-1)\Lambda \int_0^\infty U^{p^*-2}f^2r\,dr.
    \end{aligned}
  \]

  By the nondegeneracy theorem of Pistoia--Vaira~\cite{PV21},
  \begin{equation}
    \label{sec4edit:classical-kernels}
    \ker Q_{\rm Sob,p}^{(0)} = \operatorname{span}\{U,Z_0\}, \qquad \ker Q_{\rm Sob,p}^{(1)} = \operatorname{span}\{U'\}.
  \end{equation}
  Also,
  \[
    Q_{\rm Sob,p}^{(\ell)}(f)>0 \qquad \text{for every }f\ne0 \text{ and every }\ell\ge2.
  \]
  See also Figalli--Neumayer~\cite[Proposition~3.1]{FN19}, Figalli--Zhang~\cite{FZ22}, and
  \cite[Subsection~5.2]{FLZ26B}.

  For \(\xi=(\cos\alpha,\sin\alpha)\), set
  \[
    d_{\ell,p} = \frac1{2\pi} \int_0^{2\pi} |\cos\vartheta|^p\cos(\ell\vartheta)\,d\vartheta.
  \]
  The Funk--Hecke formula on \(\mathbb S^1\) is
  \begin{equation}
    \label{eq:F-identity}
    \frac1{2\pi} \int_0^{2\pi} |\cos(\vartheta-\alpha)|^p Y_{\ell,m}(\vartheta)\,d\vartheta = d_{\ell,p}Y_{\ell,m}(\alpha), \qquad m=1,2;
  \end{equation}
  see \cite[Theorem~1.2.9 and Section~1.6.1]{DX13}.

  By the \(\pi\)-periodicity of \(|\cos\vartheta|^p\), \(d_{\ell,p}=0\) for odd \(\ell\). For
  \(\ell=2k\), the beta integral gives
  \begin{equation}
    \label{eq:even-FH-coefficients-n=2}
    \frac{d_{2k,p}}{d_{0,p}} = \frac{ \Gamma\left(\frac p2+1\right)^2 }{ \Gamma\left(\frac p2+k+1\right) \Gamma\left(\frac p2-k+1\right) } = \frac{ (-1)^k\left(-\frac p2\right)_k }{ \left(1+\frac p2\right)_k }.
  \end{equation}

  For \(\ell\ge1\), set
  \[
    R_\ell(f) = 2\pi \int_0^\infty |U'|^{p-2}U' \left( f' + \frac{\ell^2}{p}\frac fr \right)r\,dr, \qquad \beta_{\ell,p} = \alpha_{2,p} \frac{\frac2p+1}{pA_0} p^2d_{\ell,p}^2.
  \]
  By \eqref{eq:F-identity} and \cite[Subsections~5.3--5.4]{FLZ26B},
  \begin{equation}
    \label{sec4edit:affine-Fourier-decomposition}
    \begin{aligned}
      Q_{\rm aff,p}(\phi) &= Q_{\rm Sob,p}^{(0)}(f_{0,1}) + \sum_{\ell=1}^{\infty} \sum_{m=1}^{2} \left[ Q_{\rm Sob,p}^{(\ell)}(f_{\ell,m}) - \beta_{\ell,p}R_\ell(f_{\ell,m})^2 \right].
    \end{aligned}
  \end{equation}
  In particular, \(Q_{\rm aff,p}^{(\ell)}=Q_{\rm Sob,p}^{(\ell)}\) for every odd \(\ell\).

  Set \(h(r)=rU'(r)\) and \(I=\int_0^\infty|U'|^p r\,dr\). For \(1<p<2\) and \(\ell\ge2\), the radial
  identities in \cite[Lemmas~5.1--5.3]{FLZ26B} give \(R_\ell(f)=\frac1p Q_{\rm Sob,p}^{(\ell)}(h,f)\)
  and
  \begin{equation}
    \label{eq:dual-norm-R-n=2}
    \begin{aligned}
      \|R_\ell\|^2_{(Q_{\rm Sob,p}^{(\ell)})^{-1}} &= \sup_{f\ne0} \frac{R_\ell(f)^2} {Q_{\rm Sob,p}^{(\ell)}(f)} = \frac{2\pi}{p^2} \left( \ell^2-(2-p) \right)I.
    \end{aligned}
  \end{equation}
  Affine invariance gives \(\beta_{2,p}=\frac{p^2}{2\pi(p+2)I}\). Hence
  \[
    \begin{aligned}
      Q_{\rm aff,p}^{(2)}(f) &= Q_{\rm Sob,p}^{(2)}(f) - \beta_{2,p}R_2(f)^2 = Q_{\rm Sob,p}^{(2)}(f) - \frac{ \left| Q_{\rm Sob,p}^{(2)}(f,h) \right|^2 }{ Q_{\rm Sob,p}^{(2)}(h) } \ge0,
    \end{aligned}
  \]
  with equality if and only if \(f\in\operatorname{span}\{h\}\).

  For \(B=B^T\) with \(\operatorname{tr}B=0\), \(x\cdot B\nabla U=rU'(r)\,\theta\cdot B\theta\). Since
  \(\theta\cdot B\theta\in\mathcal H_2(\mathbb S^1):
  =\operatorname{span}\{\cos(2\vartheta),\sin(2\vartheta)\}\) is the space of spherical harmonics of
  degree two on \(\mathbb S^1\),
  \begin{equation}
    \label{sec4edit:degree-two-kernel}
    \ker Q_{\rm aff,p}^{(2)} = \operatorname{span} \left\{ rU'(r)\cos(2\vartheta),\, rU'(r)\sin(2\vartheta) \right\}.
  \end{equation}

  Set \(\nu_k = \beta_{2k,p} \|R_{2k}\|^2_{(Q_{\rm Sob,p}^{(2k)})^{-1}}\), \(k\ge1\). Then \(\nu_1=1\),
  and \eqref{eq:even-FH-coefficients-n=2} and \eqref{eq:dual-norm-R-n=2} give
  \[
    \nu_k = \frac{ (2k)^2-(2-p) }{ p+2 } \prod_{j=1}^{k-1} \left( \frac{j-\frac p2} {j+1+\frac p2} \right)^2
  \]
  and hence
  \[
    \frac{\nu_{k+1}}{\nu_k} = \frac{ (2k+2)^2-(2-p) }{ (2k)^2-(2-p) } \left( \frac{k-\frac p2} {k+1+\frac p2} \right)^2.
  \]
  For \(1<p<2\) and \(k\ge1\),
  \[
    1-\frac{\nu_{k+1}}{\nu_k} = \frac{ 4(2k+1) \left( 4k^2p+4kp-p-2 \right) }{ (2k+p+2)^2 (4k^2+p-2) } >0.
  \]
  Thus \(\nu_{k+1}<\nu_k\), with \(\nu_2=\frac{p+14}{p+2}\left(\frac{p-2}{p+4}\right)^2<1\). Let
  \(\eta_{2,p}=1-\nu_2>0\). Then \(\nu_k\le1-\eta_{2,p}\) for \(k\ge2\), and \eqref{eq:dual-norm-R-n=2}
  gives \(\beta_{2k,p}R_{2k}(f)^2\le\nu_kQ_{\rm Sob,p}^{(2k)}(f)\). Hence
  \begin{equation}
    \label{sec4edit:higher-even-positive}
    Q_{\rm aff,p}^{(2k)}(f) \ge \eta_{2,p}Q_{\rm Sob,p}^{(2k)}(f), \qquad k\ge2.
  \end{equation}
  Therefore, by \eqref{sec4edit:classical-kernels}, \eqref{sec4edit:affine-Fourier-decomposition},
  \eqref{sec4edit:degree-two-kernel}, and \eqref{sec4edit:higher-even-positive}, we get
  \eqref{eq:affine-kernel-n=2}.
\end{proof}

\noindent\textbf{Conflict of interest:} Authors state no conflict of interest.

\noindent\textbf{Data Availability Statement:} Data sharing is not applicable to this article as no
datasets were generated or analysed during the current study.

\noindent{\bf Acknowledgement.}
G-D.~Li was supported by NSFC (No.12561019). The authors acknowledge the use of AI tools. All mathematical arguments
and proofs in the final manuscript were checked and written by the authors.

\providecommand{\href}[2]{#2}
\providecommand{\arxiv}[1]{\href{http://arxiv.org/abs/#1}{arXiv:#1}}
\providecommand{\url}[1]{\texttt{#1}}
\providecommand{\urlprefix}{URL }

\end{document}